\documentclass[a4paper,10pt]{amsart}
\usepackage{amssymb,amsmath,amsthm, bbm}
\usepackage{fullpage}
\usepackage{hyperref}
\usepackage{eucal}
\usepackage{color}
\usepackage{stackrel}
\usepackage{graphicx}
\usepackage{mathrsfs}
\usepackage{float}
\usepackage[justification=centering,margin=2cm]{caption}
\newcommand{\ie}{\emph{i.e.}}
\newcommand{\eg}{\emph{e.g.}}
\newcommand{\cf}{\emph{cf.}}

\newcommand{\Real}{\mathbb{R}}

\newcommand{\Dom}{\mathsf{D}}

\newcommand{\arctanh}{\mathop{\mathrm{arctanh}}\nolimits}

\newcommand{\eps}{\varepsilon}
\newcommand{\sii}{L^2}

\newcommand{\der}{\mathrm{d}}

\newtheorem{thm}{Theorem}[section] 
\newtheorem{Theorem}[thm]{Theorem}
\newtheorem{Proposition}[thm]{Proposition}
\newtheorem{Corollary}[thm]{Corollary}
\newtheorem{Conjecture}{Conjecture}
\newtheorem{Question}{Question}

\theoremstyle{definition}
\newtheorem{Remark}[thm]{Remark}
\begingroup
    \makeatletter
    \@for\theoremstyle:=definition,remark,plain\do{%
        \expandafter\g@addto@macro\csname th@\theoremstyle\endcsname{%
            \addtolength\thm@preskip\parskip
            }%
        }
\endgroup
\numberwithin{equation}{section}
\def\OMIT#1{}
\usepackage[normalem]{ulem}
\definecolor{DarkGreen}{rgb}{0,0.5,0.1} 

\newcommand\soutD{\bgroup\markoverwith
{\textcolor{DarkGreen}{\rule[.5ex]{2pt}{1pt}}}\ULon}

\definecolor{darkblue}{rgb}{0.2,0.2,0.6}
\definecolor{darkgreen}{rgb}{0.2,0.6,0.2}

\newcommand\soutM{\bgroup\markoverwith
{\textcolor{blue}{\rule[.5ex]{2pt}{1pt}}}\ULon}
\newcommand{\Hm}[1]{\leavevmode{\marginpar{\tiny%
$\hbox to 0mm{\hspace*{-0.5mm}$\leftarrow$\hss}%
\vcenter{\vrule depth 0.1mm height 0.1mm width \the\marginparwidth}%
\hbox to
0mm{\hss$\rightarrow$\hspace*{-0.5mm}}$\\\relax\raggedright #1}}}

\def\aa{\alpha}
\def\dR{{\mathbb R}}
\newcommand\Omg{\Omega}
\def\lm{\lambda}
\def\p{\partial}
\def\cU{{\mathcal U}}
\begin{document}
%
\title[Reverse spectral isoperimetric inequality for a triangle]{Reverse isoperimetric inequality for the lowest Robin eigenvalue of a triangle}
\author[D.~Krej\v{c}i\v{r}\'{i}k]{David Krej\v{c}i\v{r}\'{i}k}
\address{(D.~Krej\v{c}i\v{r}\'{i}k) Department of Mathematics\\ Faculty of Nuclear Sciences and Physical
	Engineering\\
	Czech Technical University in Prague\\ Trojanova 13, 120 00, Prague, Czech
	Republic\\ \linebreak
	{E-mail: david.krejcirik@fjfi.cvut.cz}}
\author[V.~Lotoreichik]{Vladimir Lotoreichik}
\address{(V.~Lotoreichik)
	Department of Theoretical Physics\\
	Nuclear Physics Institute, Czech Academy of Sciences, 
	25068 \v{R}e\v{z}, Czech Republic\\
	E-mail: {lotoreichik@ujf.cas.cz }
}

\author[T.~Vu]{Tuyen Vu}
\address{(T.~Vu) Department of Mathematics\\ Faculty of Nuclear Sciences and Physical
	Engineering\\
	Czech Technical University in Prague\\ Trojanova 13, 120 00, Prague, Czech
	Republic\\
	E-mail: {thibichtuyen.vu@fjfi.cvut.cz}}
\date{23 February 2023}
\begin{abstract}
We consider the Laplace operator on a triangle,
subject to attractive Robin boundary conditions.
We prove that the equilateral triangle is a local maximiser of 
the lowest eigenvalue among all triangles of 
a given area provided that the negative boundary parameter
is sufficiently small in absolute value,
with the smallness depending on the area only.
Moreover, using various trial functions, we obtain sufficient conditions 
	for the global optimality of the equilateral triangle under fixed area constraint
	in the regimes of small and large couplings. 
	We also discuss the constraint of fixed perimeter.
\end{abstract} 
\maketitle

\section{Introduction}
%
Given a bounded open connected set $\Omega\subset\Real^d$
of dimension $d\ge2$ and with Lipschitz boundary~$\partial\Omega$,
consider the Robin eigenvalue problem
\begin{equation}\label{bvp}
\left\{
\begin{aligned}
  -\Delta u &= \lambda u 
  & \mbox{in} \quad \Omega \,,
  \\
  \frac{\partial u}{\partial n} + \alpha \;\! u &= 0 
  & \mbox{on} \quad \partial\Omega \,,
\end{aligned}
\right.
\end{equation}
where~$\alpha$ is a real parameter and~$n$ is the outward unit normal of~$\Omega$. 
Let us arrange the corresponding eigenvalues in a non-decreasing sequence by
$\{\lambda_k^\alpha(\Omega)\}_{k=1}^\infty$,
where each eigenvalue is repeated according to its multiplicity.
For the lowest eigenvalue, one has the variational characterisation
\begin{equation}\label{variational}
  \lambda_1^\alpha(\Omega) =
  \inf_{\stackrel[u\not=0]{}{u \in H^1(\Omega)}}
  \frac{\displaystyle
  \int_\Omega |\nabla u|^2 + \alpha \int_{\partial\Omega} |u|^2}
  {\displaystyle
  \int_\Omega |u|^2}
  \,,
\end{equation}
where the boundary integral is understood in the sense of traces.
By convention, we also include the Dirichlet case $\alpha=+\infty$,
where the space of test functions is $H_0^1(\Omega)$
and the boundary integral is not present.

The present paper is primordially motivated 
by the following broad question
about the validity of spectral isoperimetric inequalities
for the Robin Laplacian.

\begin{Question}\label{Conj}
For every $\alpha \in (-\infty,+\infty]$, 
does one have
\begin{equation}\label{Conj.eq}
  \frac{\lambda_1^\alpha(\Omega)}{\lambda_1^\alpha(\Omega^*)} \geq 1
  \,,
\end{equation}
where~$\Omega^*$ is the ball of the same volume as~$\Omega$?
\end{Question}

Since $\lambda_1^0(\Omega) = 0$ for any domain~$\Omega$,
the Neumann case $\alpha = 0$ should be understood 
through the limit $\alpha \to 0$. 
It is easy to verify the asymptotics
$
  \lambda_1^\alpha(\Omega) = \alpha |\partial\Omega| / |\Omega|
  + O(\alpha^2)
$
as $\alpha \to 0$, 
where~$|\Omega|$ and $|\partial\Omega|$ denote 
the volume of~$\Omega$ and the $(d-1)$-dimensional 
Hausdorff measure of $\partial\Omega$, respectively.
Then the inequality~\eqref{Conj.eq} for $\alpha=0$ can be interpreted
as the classical (purely geometric) isoperimetric inequality
$|\partial\Omega| \geq |\partial\Omega^*|$
whose validity is well known.
Indeed, the ball has the smallest surface area
among all domains of the same volume. 

The Dirichlet case $\alpha=+\infty$ is also well known
and customarily referred to as the Faber--Krahn inequality.
It was conjectured by Lord Rayleigh in his celebrated monograph~\cite{R}
and rigorously proved fifty years later in~\cite{4} and~\cite{8}. 
In dimension $d=2$, the usual interpretation is that
among all membranes of equal area and fixed edges,
the circular membrane emits the lowest fundamental tone.

For other values of~$\alpha$, 
the history of Question~\ref{Conj} is much more recent.
The repulsive case $\alpha>0$ is known as 
the Bossel--Daners inequality~\cite{2,3}
(see also \cite{Bucur-Giacomini_2015} for an alternative proof).
Here the classical interpretation in dimension $d=2$
is that among all membranes of equal area and with elastically supported edges
(think about the membrane attached by springs),
the circular membrane emits the lowest fundamental tone.
In all dimensions, we may rely on a quantum-mechanical interpretation
that among all resonators of equal volume
and with a strongly localised positive potential along the boundary, 
the spherical resonator has the smallest ground-state energy.

In summary, the spectral isoperimetric inequality
$\lambda_1^\alpha(\Omega) \geq \lambda_1^\alpha(\Omega^*)$ holds
for all positive~$\alpha$ (including the Dirichlet case).
What happens for negative~$\alpha$?
Since $\lambda_1^\alpha(\Omega)$ is negative if (and only if) $\alpha < 0$,
property~\eqref{Conj.eq} means the reverse inequality
$\lambda_1^\alpha(\Omega) \leq \lambda_1^\alpha(\Omega^*)$ 
for negative~$\alpha$. 
In other words,
the ball plays the role of the \emph{maximiser} in the attractive case,
and so one usually speaks about 
the \emph{reverse} spectral isoperimetric inequality
in this regime.

While springs with a negative force constant
 are perhaps less intuitive,
the quantum-mechanical interpretation is meaningful
for the attractive case $\alpha<0$ as well. 
(There are also alternative interpretations in acoustics~\cite{Morse}
and superconductivity~\cite{GS07}.)
Consider quantum resonators of equal volume
and with a strongly localised negative potential along the boundary.
Is it true that the spherical resonator has the largest ground-state energy?
It turns out that the question of validity of~\eqref{Conj.eq} 
for $\alpha < 0$ constitutes a hot open problem in spectral geometry
known as the Bareket conjecture~\cite{1}
and it still remains open.
In fact, this conjecture does not hold 
for multiply connected domains~\cite{6}
and without convexity assumption in space dimensions $d\ge3$~\cite{FNT16}.
So the open question is precisely whether it holds within
the class of simply connected domains in two dimensions 
\cite[Conj.~2]{AFK}
and the class of convex domains in higher space dimensions
\cite[Conj.~3]{AFK}. 

The fixed volume
is not the only geometric constraint of interest.  The counterpart of~\eqref{Conj.eq} in the attractive case $\alpha < 0$ 
under fixed area of the boundary constraint is settled for $d = 2$ in~\cite{AFK} and for $d\ge3$  in the class of convex domains in~\cite{BFNT} (see also~\cite{V20}). In this setting, it is conjectured that convexity of the domain for $d \ge 3$ can be dropped~\cite[Conj.~4]{AFK}.

Surprisingly, the analogous questions of spectral optimisation
in the exterior of bounded sets are simpler 
and have been resolved recently \cite{KL1,KL2}.
 
Our goal in the present paper is to 
consider a ``discrete'' version of Question~\ref{Conj}
in the sense that we restrict ourselves 
to a special class of planar domains.
Namely, we address the following conjecture, 
which also appears in the list of open problems~\cite[Conj.~1.3]{KLL}.

\begin{Conjecture}\label{Conj.triangle}
Let $\Omega\subset\Real^2$ be any triangle.
For every $\alpha \in (-\infty,0) \cup (0,+\infty)$, one has
\begin{equation}\label{Conj.eq.triangle}
  \frac{\lambda_1^\alpha(\Omega)}{\lambda_1^\alpha(\Omega^*)} \geq 1
  \,,
\end{equation}
where~$\Omega^*$ is the equilateral triangle of the same area as~$\Omega$.  
\end{Conjecture}

The Dirichlet case $\alpha = +\infty$ is excluded here 
because it can be settled by the Steiner symmetrisation,
see \cite[Sec.~7.4]{Polya-Szego} or \cite[Thm.~3.3.3]{H06}. 
(An analogous conjecture for quadrilaterals 
also holds in the Dirichlet case, but polygons with more sides
still constitute an interesting open problem in spectral geometry,
see \cite{Bucur-preprint} for the latest developments.)
Similarly, the Neumann case $\alpha=0$, 
when the left-hand side of~\eqref{Conj.eq.triangle}
is interpreted as the limit of the quotient as $\alpha \to 0$,
reduces to the well-known purely geometric statement that 
the equilateral triangle has the smallest perimeter among all triangles
of the same area. 
(A relevant spectral-geometric question 
for the Neumann boundary condition is to optimise 
the first non-zero eigenvalue: it is maximised by the equilateral triangle
both for the area or perimeter constraints~\cite{LS}.)
Apart from the special situation of Dirichlet and Neumann boundary conditions,
both cases $\alpha < 0$ and $\alpha > 0$ (finite) 
of Conjecture~\ref{Conj.triangle} remain open. 
In this paper we provide a partial answer for~$\alpha$ negative,
when~\eqref{Conj.eq.triangle} means the reverse spectral
isoperimetric inequality
$\lambda_1^\alpha(\Omega) \leq \lambda_1^\alpha(\Omega^*)$,
\ie~the equilateral triangle is a maximiser.

As the main result, we prove that the equilateral triangle is 
a \emph{local} maximiser provided that~$\alpha$ is negative 
and sufficiently small in absolute value.
That is, \eqref{Conj.eq.triangle} holds asymptotically
in the attractive case,
whenever the triangle~$\Omega$ is close 
to the equilateral triangle~$\Omega^*$ 
in the sense of Hausdorff distance (after possible congruences).
(Note that the Hausdorff distance of \emph{triangles}~$\Omega$ and~$\Omega^*$ 
tends to zero if, and only if, the area $|\Omega \setminus \Omega^*|$
tends to zero.)
\begin{Theorem}\label{Thm.local}
Let~$\Omega^*$ be an equilateral triangle
and let~$\Omega$ denote any triangle of the same area.
There exists a negative number $\alpha_0$ 
depending solely on the fixed area such that, 
for all $\alpha \in [\alpha_0,0)$,
\begin{equation*} 
 	\frac{\lambda_1^\alpha(\Omega)}{\lambda_1^\alpha(\Omega^*)} \geq 1
  \,,
\end{equation*}
provided that $|\Omega\setminus\Omega^*|$ is sufficiently small (with the smallness depending also on the area $|\Omega|=|\Omega^*|$ fixed).
\end{Theorem}

Moreover, we provide an explicit estimate for~$\alpha_0$.
We also prove that the same local optimality result holds 
under the fixed perimeter constraint. 
We leave as an open problem whether the restriction 
on the smallness of~$|\alpha|$ is necessary for the validity 
of the local maximisation result. 

In order to prove the local optimality of the equilateral triangle,
we show that the first-order and also the second-order mixed derivatives of the lowest eigenvalue with respect to geometric parameters characterising the triangle vanish and we obtain upper bounds on the second-order partial derivatives. The rest of the analysis reduces to identifying conditions on the boundary parameter under which these upper bounds are negative. In this argument, we take the advantage of rewriting the quadratic form of the Robin Laplacian on a general triangle into the quadratic form on an equilateral triangle with the geometry transferred into coefficients in the differential expression and into boundary parameters. Upon such a transform all the quadratic forms are defined in the same Hilbert space and depend analytically on the geometric parameters. In the computation of the derivatives of the lowest eigenvalue we exploit the theory of self-adjoint holomorphic families of type (B) (see~\cite[Chap.~VII]{7}). Another ingredient of the argument is that the Robin eigenvalue problem on an equilateral triangle is explicitly solvable~\cite{5,9}. 

Our next series of results is about 
the \emph{global} validity of Conjecture~\ref{Conj.triangle},
still in the attractive case.
In particular, we establish the conjecture for 
negative~$\alpha$ which is either small or large in absolute value. 
\begin{Theorem}\label{Thm.global}
Let~$\Omega^*$ be an equilateral triangle
and let~$\Omega$ denote any triangle of the same area.
There exist negative numbers $\alpha_1 \leq \alpha_2$ such that,
for all $\alpha \in (-\infty,\alpha_1] \cup [\alpha_2,0)$,
\begin{equation*}  
  \lambda_1^\alpha(\Omega) \leq \lambda_1^\alpha(\Omega^*)
  \,.
\end{equation*}
\end{Theorem}
Here the constants $\alpha_1$ and $\alpha_2$
\emph{a priori} depend on the geometry of~$\Omega$.
In fact, as stated in the present form, this result follows
rather straightforwardly by known eigenvalue asymptotics
for $\alpha \rightarrow 0$ and $\alpha \rightarrow -\infty$ (see~\cite[Chap. 4, Eq. (4.12)]{AH} and ~\cite[Thm 2.3]{11}, respectively).   
However, more uniform and explicit results can be found in the theorems below.
A suitable value for the
constant $\alpha_1$ can be deduced from
the geometric condition~\eqref{eq:condition} in Theorem~\ref{thm:large}. In particular, the constant $\alpha_1$ can be chosen arbitrarily close to zero provided that the smallest angle of the triangle $\Omg$ is sufficiently small.   For the constant $\alpha_2$ we have no explicit condition, but we are able to show in Theorem~\ref{Thm.small} that one can choose a universal value for the constant $\alpha_2$ suitable for a class of triangles, which is, roughly speaking, characterised by the maximal possible degree of the deviation from the equilateral triangle. 
The full proof of Conjecture~\ref{Conj.triangle} 
in the attractive case would require 
to show that $\alpha_1 = \alpha_2$, 
which constitutes an interesting open problem.

Theorem~\ref{Thm.global} and its variants are established 
via a combination of the variational characterisation of 
the lowest eigenvalue and various trial functions.
In particular, for small $|\aa|$, it is sufficient 
to transplant the eigenfunction of the equilateral triangle onto the general triangle via a geometric transformation. 
For large $|\aa|$, a suitable trial function is obtained 
by constant functions
or by a truncation of the eigenfunction of the Robin Laplacian on a sector.
The regions where these trial functions work 
for the proof of Conjecture~\ref{Conj.triangle}
as well as their limitations are analysed numerically. In particular, we check numerically the validity of the inequality in Conjecture~\ref{Conj.triangle} for all $\alpha < 0$ for a class of triangles with high deviation from the equilateral triangle.
 
The paper is organised as follows.
In Section~\ref{sec:prelim} we rigorously define the Robin Laplacian on a general triangle, construct the transform which reduces the spectral problem into an equivalent one on an equilateral triangle,
and provide a characterisation of the lowest eigenvalue 
and the respective eigenfunction of the equilateral triangle.
In Section~\ref{Sec.derivative} 
we compute the derivatives of the lowest eigenvalue 
on a general triangle with respect to the geometric parameters
distinguishing the triangle from the equilateral triangle;
in particular, we establish Theorem~\ref{Thm.local}. 
In Sections~\ref{sec:small} and~\ref{sec:large}
we obtain sufficient conditions for the validity of 
the isoperimetric inequality for small and large~$|\aa|$, respectively;
in particular, we establish Theorem~\ref{Thm.global}.

\section{Preliminaries}\label{sec:prelim}
\subsection{The Robin Laplacian on a triangle}
We understand the Robin eigenvalue problem~\eqref{bvp}  
on a general Lipschitz domain $\Omega\subset\dR^d$ with $d\ge2$
as the spectral problem in the Hilbert space $\sii(\Omega)$
for the self-adjoint operator
\begin{equation}\label{operator}
H_\alpha u := -\Delta u \,, \qquad
\Dom(H_\alpha) := \left\{u \in H^1(\Omega): \ \Delta u \in L^2(\Omega)
\quad \& \quad
\frac{\partial u}{\partial n} + \alpha \;\! u = 0 
\quad \mbox{on} \quad \partial\Omega
\right\} 
.
\end{equation}
Here the normal derivative should be viewed 
as a distribution in the Sobolev space $H^{-1/2}(\partial\Omega)$.
The operator~$H_\alpha$ is associated with a
closed, densely defined, symmetric and semi-bounded quadratic form 
in $L^2(\Omg)$ given by
\begin{equation}\label{eq:form}
h_\alpha[u] := \int_\Omega |\nabla u|^2 + \alpha \int_{\partial\Omega} |u|^2
\,, \qquad
\Dom(h_\alpha) := H^1(\Omega)
\,.
\end{equation}
Here the value of~$u$ in the boundary integral 
is regarded as the trace of the function $u \in H^1(\Omega)$.

\subsection{Geometric setting} \label{sec:geom}
In this subsection we reduce the spectral problem for the Robin Laplacian on a general triangle to the spectral problem for a certain second-order differential expression with constant coefficients on an equilateral triangle, in which the geometric parameters enter into the coefficients of the differential expression and into the boundary parameters on the edges of the triangle. In order to perform this transformation we construct a suitable unitary transform.

We parameterise general triangles which have the fixed area $S>0$ 
by parameters $a \in \Real$ and $c > 0$.
Let~$\Omega_{a,c}$ be the triangle with the vertices
$(-c,0)$, $(c,0)$ and $(a,b)$ 
with $b := \tfrac{S}{c}$. 
The special choice $a_0:=0$ and $c_0:=\sqrt{\tfrac{S}{\sqrt{3}}}$
leads to the equilateral triangle $\Omega_0 := \Omega_{a_0,c_0}$.
The area and the perimeter of the triangles $\Omg_{a,c}$ and $\Omg_0$ are given by
\begin{equation}\label{explicit}
\begin{aligned}
  |\Omega_{a,c}| &= b c=S = |\Omega_0|
  \qquad \mbox{(independent of~$a,c$)} \,,
  \\
  |\partial\Omega_{a,c}| &= 
  \left( 2c + \sqrt{\frac{S^2}{c^2}+(a-c)^2} + \sqrt{\frac{S^2}{c^2}+(a+c)^2} \right) 
   \,,
  \qquad
  |\partial\Omega_0| = 6 \mbox{$\sqrt{\frac{S}{\sqrt{3}}}$}
  \,.
\end{aligned}
\end{equation}

The boundary of~$\Omega_{a,c}$ consists of three sides
\begin{equation*}
  \partial\Omega_{a,c} = \Gamma_{a,c}^{(0)} \cup \Gamma_{a,c}^{(1)} \cup \Gamma_{a,c}^{(2)}
  \,,
\end{equation*}
where ($a \not= \pm c$)
\begin{equation*}
\begin{aligned}
  \Gamma_{a,c}^{(0)} &:= \left\{ (x,0) : \ 
  |x| \leq \mbox{$\sqrt{\frac{S}{\sqrt{3}}}$} \right\}
  , \\
  \Gamma_{a,c}^{(1)} &:= \left\{ 
  \left(x, \frac{S}{c(a+c)} x+\frac{S}{a+c} \right) : \ 
  x \in [\min\{-c,a\},\max\{-c,a\}] \right\}
  , \\
  \Gamma_{a,c}^{(2)} &:= \left\{ 
  \left(x,\frac{S}{c(a-c)} x+\frac{S}{c-a} \right) : \ 
  x \in [\min\{c,a\},\max\{c,a\}] \right\}
  .
\end{aligned}
\end{equation*}
The cases $a = \pm c$ 
correspond to right-angled triangles and
require a separate consideration; we will comment more on them in Remark~\ref{rem:a=pmc} below.
Consequently, given any function $u \in H^1(\Omega_{a,c})$, 
one has ($a \not= \pm c$)
\begin{equation*}
\begin{aligned}
  \|u\|_{\sii(\Gamma_{a,c}^{(0)})}^2
  &= \int_{-c}^{c} |u(x,0)|^2 \, \der x
  \,,
  \\
  \|u\|_{\sii(\Gamma_{a,c}^{(1)})}^2
  &= \frac{\sqrt{\frac{S^2}{c^2}+(a+c)^2}}{a+c}
  \int_{-c}^{a} 
  \left|u\left(x,\frac{S}{c(a+c)} x+\frac{S}{a+c} \right)\right|^2
  \, \der x
  \,,
  \\
  \|u\|_{\sii(\Gamma_{a,c}^{(2)})}^2
  &= \frac{\sqrt{\frac{S^2}{c^2}+(a-c)^2}}{c-a}
  \int_{a}^{c} 
  \left|u\left(x,\frac{S}{c(a-c)} x+\frac{S}{c-a} \right)\right|^2
  \, \der x
  \,.
\end{aligned}
\end{equation*}
%
Using Fubini's theorem, these parameterisations can also be used 
to compute $\|u\|_{\sii(\Omega_{a,c})}$.
Alternatively,
\begin{equation*}
  \|u\|_{\sii(\Omega_{a,c})}^2
  = \int_0^{\frac{S}{c}} 
  \int_{\frac{c(a+c)}{S} y -c}
  ^{\frac{c(a-c)}{S} y +c}
  |u(x,y)|^2 \, \der x \, \der y \,.
\end{equation*}

An identification of the triangles~$\Omega_{a,c}$ and~$\Omega_0$
is obtained by the diffeomorphism
\begin{equation*}
  \mathscr{L}_{a,c}: \Omega_0 \to \Omega_{a,c}:
  \left\{
  (x,y) \mapsto \left(
  \mbox{$\frac{c}{\sqrt{\frac{S}{{\sqrt{3}}}}}$} x+\mbox{$\frac{a}{\sqrt{\sqrt{3}S}}$} y,\mbox{$\frac{\sqrt{\frac{S}{\sqrt{3}}}}{c}$} y 
  \right)
  \right\}
  \,.
\end{equation*}
The corresponding metric reads
\begin{equation*}
  G_{a,c} 
  := \nabla\mathscr{L}_{a,c} \cdot (\nabla\mathscr{L}_{a,c})^T =
  \begin{pmatrix}
    \frac{c^2}{\frac{S}{\sqrt{3}}} & \frac{ac}{S} \\
    \frac{ac}{S} & \frac{a^2}{\sqrt{3}S}+\frac{S}{\sqrt{3}c^2}
  \end{pmatrix}
  \,,
  \qquad
  |G_{a,c}| := \det(G_{a,c}) = 1
  \,, 
\end{equation*}
where the dot~$\cdot$ stands for the matrix multiplication.
The inverse metric reads
\begin{equation*}
  G_{a,c}^{-1} =
  \begin{pmatrix}
    \frac{a^2c^2+S^2}{\sqrt{3}c^2S} & -\frac{ac}{S} \\
    -\frac{ac}{S} & \frac{\sqrt{3}c^2}{S}
  \end{pmatrix}
  \,.
\end{equation*}

The diffeomorphism~$\mathscr{L}_{a,c}$ induces the unitary map
\begin{equation*}
  U_{a,c} : \sii(\Omega_{a,c}) \to \sii(\Omega_0) :
  \left\{ 
  u \mapsto u \circ \mathscr{L}_{a,c}
  \right\}
  \,.
\end{equation*}
We set $\hat{H}_{\alpha,a,c} := U_{a,c} H_\alpha U_{a,c}^{-1}$,
where $H_\alpha$ is the operator~\eqref{operator}
for $\Omega = \Omega_{a,c}$.
Taking $u \in H^1(\Omega_{a,c})$ 
and denoting $\psi:=u \circ \mathscr{L}_{a,c} \in H^1(\Omega_0)$,
one has
\begin{equation*}
\begin{aligned}
  \|u\|_{\sii(\Omega_{a,c})}^2 &= \|\psi\|_{\sii(\Omega_0)}^2 
  \,,
  \\
  \|\nabla u\|_{\sii(\Omega_{a,c})}^2 &=
  \tfrac{S}{\sqrt{3}c^2}
  \left\|\partial_1\psi\right\|_{\sii(\Omega_0)}^2 + 
  \left\|
  c\sqrt{\tfrac{\sqrt{3}}{S}}\partial_2\psi - \tfrac{a}{\sqrt{\sqrt{3}S}} \partial_1\psi
  \right\|_{\sii(\Omega_0)}^2
  \,,
  \\
  \|u\|_{\sii(\Gamma_{a,c}^{(0)})}^2
  &= \int_{-{ c}}^{{ c}}
  |{ u}(x,0)|^2 \, \der x
  = \mbox{$\frac{c}{\sqrt{\frac{S}{\sqrt{3}}}}$}\|\psi\|_{\sii(\Gamma_0^{(0)})}^2
  \,,
  \\
  \|u\|_{\sii(\Gamma_{a,c}^{(1)})}^2
  &= \frac{\sqrt{\frac{S^2}{c^2}+(a+c)^2}}{a+c}
  \int_{-c}^{a} 
  \left|u\left(x,\frac{S}{c(a+c)} x+\frac{S}{a+c} \right)\right|^2
  \, \der x
  =\mbox{$ \sqrt{\frac{\sqrt{3}}{S}}$} \frac{\sqrt{c^2(a+c)^2+S^2}}{2c} \|\psi\|_{\sii(\Gamma_0^{(1)})}^2
  \,,
  \\
  \|u\|_{\sii(\Gamma_{a,c}^{(2)})}^2
  &= \frac{\sqrt{\frac{S^2}{c^2}+(a-c)^2}}{c-a}
  \int_{a}^{c} 
  \left|u\left(x,\frac{S}{c(a-c)} x+\frac{S}{c-a} \right)\right|^2
  \, \der x
  = \mbox{$ \sqrt{\frac{\sqrt{3}}{S}}$} \frac{\sqrt{c^2(a-c)^2+S^2}}{2c} \|\psi\|_{\sii(\Gamma_0^{(2)})}^2
  \,,
\end{aligned}
\end{equation*}
where we used the abbreviation $\Gamma_0^{(k)} := \Gamma_{a_0,c_0}^{(k)}$.
Consequently, $\hat{H}_{\alpha,a,c}$~is the operator in
the ($a,c$)-independent Hilbert space $\sii(\Omega_0)$
associated with the quadratic form
$\hat{h}_{\alpha,a,c}[\psi] := h_\alpha[U_{a,c}^{-1}\psi]$
with domain $\Dom(\hat{h}_{\alpha,a,c}) := U_{a,c}( \Dom(h_\alpha))$. 
One has
\begin{equation}\label{form.hat}
\begin{aligned}
  \hat{h}_{\alpha,a,c}[\psi] &= \  \tfrac{S}{\sqrt{3} c^2}\|\partial_1\psi\|_{\sii(\Omega_0)}^2 + 
  \big\|
 c\sqrt{\tfrac{\sqrt{3}}{S}}\partial_2\psi - \tfrac{a}{\sqrt{\sqrt{3}S}} \partial_1\psi
  \big\|_{\sii(\Omega_0)}^2
  \\ 
  & + \alpha \, \tfrac{c}{\sqrt{\frac{S}{\sqrt{3}}}}\|\psi\|_{\sii(\Gamma_0^{(0)})}^2
  + \alpha \, \sqrt{\tfrac{\sqrt{3}}{S}} \frac{\sqrt{c^2(a+c)^2+S^2}}{2c} \|\psi\|_{\sii(\Gamma_0^{(1)})}^2	
  + \alpha \, \sqrt{\tfrac{\sqrt{3}}{S}} \frac{\sqrt{c^2(a-c)^2+S^2}}{2c} \|\psi\|_{\sii(\Gamma_0^{(2)})}^2, \, \\
  \mathsf{D}(\hat{h}_{\alpha,a,c})  &
   = H^1(\Omega_0)\ .
\end{aligned}
\end{equation}
Notice that, in particular, 
$\hat{h}_{\alpha,a_0,c_0}$ coincides with the quadratic form $h_\alpha$ in~\eqref{eq:form} for $\Omega = \Omega_0$.

\begin{Remark}\label{rem:a=pmc}
The formula \eqref{form.hat} remains unchanged in the situations $a=\pm c$,
which correspond to the right-angled triangle, even
though in this case $\Gamma_{a,c}^{(2)}$ for $a =c$ 
(respectively, $\Gamma_{a,c}^{(1)}$ for $a=-c$) 
must be parameterised differently. 
\end{Remark}

The reduction in this subsection is valid for all $\aa\in\dR$. However, in the following we consider the attractive case  $\alpha < 0$ only.

\subsection{The equilateral triangle}\label{ssec:equilateral}
%
In this subsection we provide a preliminary analysis for the Robin Laplacian 
on an equilateral triangle for $\aa < 0$. The lowest eigenvalue 
and the respective eigenfunction in this setting can be computed explicitly. 
In particular, the ground state is
expressed in terms of hyperbolic cosines and the respective eigenvalue is expressed in terms of solutions of a system of two transcendental equations involving the area of the triangle and the boundary parameter. The details on this analysis can be found in~\cite{9}, see also~\cite{5}.
We provide also additional computations that will be essential in the proof of the main result.

The operator $\hat{H}_{\alpha, a_0,c_0}$ is just the Robin Laplacian
on the equilateral triangle $\Omega_0 = \Omega_{a_0,c_0}$. 
The eigenfunction corresponding to its lowest eigenvalue 
$\lambda_1^\alpha(\Omega_0)$ reads
\begin{equation}\label{eq:ef}
  u_0(x,y)\! := \!
  \cosh\left(L\! -\! 2 \frac{ (L - M) y}{ \sqrt{\sqrt{3}S}}\right) 
  \!+\! 2 \cosh\left((M - L)\left(1\! -\! \frac{y}{\sqrt{\sqrt{3}S}}\right) \!+\! L\right)
  \cosh\left(\frac{(M - L) \sqrt{3} x}{\sqrt{\sqrt{3}S}} \right)
  \,.
\end{equation}
Here the numbers $L,M$ are to be determined by the equations
\begin{equation}\label{implicit}
\begin{aligned}
  2 \, (L-M) \tanh L &= -\alpha \, \mbox{$\sqrt{\sqrt{3}S}$} \, ,
  \\
  (M-L) \tanh M &= -\alpha \, \mbox{$\sqrt{\sqrt{3}S}$} \, .
\end{aligned}
\end{equation}
We remark that the eigenfunction $u_0$ is, in general, not normalised.
The eigenvalue $\lambda_0(\alpha) := \lambda_1^\alpha(\Omega_0)$ satisfies
\begin{equation}\label{ev0}
  \lambda_0(\alpha) = \frac{-4(M-L)^2}{\sqrt{3}S} \,.	
\end{equation}

Dividing the two equations of~\eqref{implicit}, 
we arrive at the identity
\begin{equation}
  \tanh M + 2 \tanh L = 0 \,.
\end{equation}
Expressing
$
  L = - \arctanh\big(\mbox{$\frac{1}{2}$} \tanh M \big)
$,
we convert~\eqref{implicit} into a unique implicit equation
\begin{equation}\label{eq:M}
  \left[ M + \arctanh\big(\mbox{$\frac{1}{2}$} \tanh M \big) \right]
  \tanh M = -\alpha\mbox{$\sqrt{\sqrt{3}S}$} \,.
\end{equation}
We introduce the notation
\[
	K: = M-L > 0.
\]	
Hence, \eqref{ev0}~reads 
\begin{equation}\label{lambda0}
  \lambda_0(\alpha) = \frac{-4K^2}{\sqrt{3}S} \,,	
\end{equation}
where $K, \alpha$ satisfy the following implicit equation
\begin{equation}\label{implicit_function}
   F(K, \alpha): = K + \arctanh\left(\mbox{$\sqrt{\sqrt{3}S}$}\dfrac{\alpha}{K}\right)+ \arctanh\left(\mbox{$\sqrt{\sqrt{3}S}$}\dfrac{\alpha}{2K}\right) =0. 
\end{equation} 

Let us further introduce the parameter
\[
	t:= \dfrac{-\alpha \, \mbox{$\sqrt{\sqrt{3}S}$}}{K}.
\]
From the second equation in~\eqref{implicit}
it is clear that $t\in(0,1)$.
By implicit derivative formula applied to~\eqref{implicit_function}, 
we find
\begin{equation*}
  \frac{\alpha}{K} K_\alpha' =  
  	\frac{\frac{t}{1-t^2} +\frac{t}{2(1-\frac{t^2}{4})}}{K+\frac{t}{1-t^2} +\frac{t}{2(1-\frac{t^2}{4})}} \, .
\end{equation*}  
Here we use the notation $f_x'$ for a partial derivative 
of a function~$f$ with respect to~$x$.
Introducing the auxiliary functions 
\[
	A(t) := \frac{t}{1-t^2}+ \frac{t}{2(1-\frac{t^2}{4})} >0
	\qquad\text{and}\qquad k(\alpha) := \dfrac{\alpha}{K} K_\alpha', 
\]	 
we find several identities which will be useful in what follows:
\begin{equation}\label{alpha}
\begin{aligned}
 k&= \frac{A}{K+A},\\
 K &= \arctanh t+ \arctanh\dfrac{t}{2},\\ \alpha & = -\mbox{$\frac{1}{\sqrt{\sqrt{3}S}}$} \, t\left(\arctanh t+ \arctanh\dfrac{t}{2}\right).
\end{aligned} 
\end{equation}

Finally, we analyse the dependence of $\lm_0(\aa)$ on the area of the triangle.
\begin{Proposition}\label{prop:monotone_S}
	The lowest Robin eigenvalue of the equilateral triangle $\lm_0(\aa)$ is a (strictly) increasing function of the area $S$.
\end{Proposition}
\begin{proof}
	First, we find from~\eqref{eq:M} and the expression for $L$ (above that equation) that
	\[
		K = -\frac{\aa\sqrt{\sqrt{3}S}}{\tanh M}.
	\]
	Hence, it follows from~\eqref{lambda0} 
	that the lowest eigenvalue can be expressed as
	\[
		\lm_0(\aa) = -\frac{4\aa^2}{\tanh^2M}.
	\]
	In order to prove that $\lm_0(\aa)$ is an increasing function of $S$ it suffices to show that $M$ is an increasing function of $S$. To this aim we analyse the transcendental equation~\eqref{eq:M} in detail.
	The right-hand side of that equation is positive and increasing in $S$. Hence, it suffices to check that the left-hand side is an increasing function of $M > 0$. The desired property immediately follows from the fact that $\tanh(x)$ and
	thus its inverse $\arctanh(x)$ are increasing functions of $x > 0$.	
\end{proof}	
%

\section{Local optimality of the equilateral triangle}\label{Sec.derivative}
%
The aim of this section is to show that under fixed area constraint the equilateral triangle is a local maximiser of the lowest Robin eigenvalue of a triangle for the negative boundary parameter lying in a neighbourhood of $0$ determined solely by the area of the triangle. We make use of the reduction of the spectral problem for the Robin Laplacian on a general triangle to an equivalent problem on an equilateral triangle constructed in Subsection~\ref{sec:geom}.
This reduction allows us to view the obtained family of respective quadratic forms as a self-adjoint holomorphic family and to compute the derivatives of the lowest eigenvalue in terms of geometric parameters. We establish that the first-order and the second-order mixed derivatives vanish and obtain upper bounds on the second-order partial derivatives in terms of certain quadratic forms evaluated on the ground state for the equilateral triangle. In the course of
the computations we encounter cumbersome intermediate expressions that we provide for convenience of the reader.  The rest of the analysis boils down to showing that these upper bounds are negative under certain assumptions on the boundary parameter. 

Since the operators~$H_\alpha$ and~$\hat{H}_{\alpha,a,c}$ are isospectral,
$\lambda_1^\alpha(\Omega_{a,c}) =: \lambda_{a,c}$ 
is the lowest eigenvalue of~$\hat{H}_{\alpha,a,c}$ as well.
It is easy to check that the operators $\hat{H}_{\alpha,a,c}$ form
a self-adjoint holomorphic family 
(of type~(B) in the sense of Kato~\cite[Sec.~VII.4]{7})
with respect to the parameters~$a$ and~$c$ (separately).
Moreover, the eigenvalue~$\lambda_{a,c}$ is simple.
Hence, $a \mapsto \lambda_{a,c}$ is a real-analytic function on~$\Real$ and $c \mapsto \lambda_{a,c}$ is a real-analytic function on~$(0, \infty)$ .
Let~$\psi_{a,c}$ denote the positive eigenfunction of~$\hat{H}_{\alpha,a,c}$
corresponding to~$\lambda_{a,c}$ such that $\|\psi_{a,c}\|_{\sii(\Omega_0)}=1$.
Then we also have that $a \mapsto \psi_{a,c}$ is a real-analytic function on~$\Real$ and $c \mapsto \psi_{a,c}$ is a real-analytic function on~$(0, \infty)$.
For the equilateral triangle,
we abbreviate $\lambda_0 :=\lambda_{a_0,c_0}$ and $\psi_0 :=\psi_{a_0,c_0}$.

The eigenvalue problem $\hat{H}_{\alpha,a,c} \psi_{a,c} = \lambda_{a,c} \psi_{a,c}$
is equivalent to the weak formulation
\begin{equation}\label{id0}
  \forall \phi \in H^1(\Omega_0) \,, \qquad
  \hat{h}_{\alpha,a,c}(\phi,\psi_{a,c}) = \lambda_{a,c} \, (\phi,\psi_{a,c})
  \,,
\end{equation}
where we abbreviate $(\cdot,\cdot) := (\cdot,\cdot)_{\sii(\Omega_0)}$.
We shall also write 
$(\cdot,\cdot)_{k} := (\cdot,\cdot)_{\sii(\Gamma_0^{(k)})}$
and similarly for the corresponding norm.
Note that the test function~$\phi$ can be assumed 
to be real-valued without loss of generality.

\subsection{The first derivative with respect to~\texorpdfstring{$a$}{a}}
Differentiating the identity~\eqref{id0} with respect to~$a$
(the corresponding derivative being denoted by the dot),
we get
\begin{equation}\label{id1}
\begin{aligned}
&\mbox{$\frac{S}{\sqrt{3}c^2}$} \, (\partial_1\phi,\partial_1\dot\psi_{a,c}) 
  + \left(\mbox{$c\sqrt{\frac{\sqrt{3}}{S}}$} \partial_2\phi - \mbox{$\frac{a}{\sqrt{\sqrt{3}S}}$} \partial_1\phi,
 \mbox{$ c\sqrt{\frac{\sqrt{3}}{S}}$} \partial_2\dot\psi_{a,c} -\mbox{$ \frac{a}{\sqrt{\sqrt{3}S}}$} \partial_1\dot\psi_{a,c}\right)\\
&\quad\mbox{$  - \frac{1}{\sqrt{\sqrt{3}S}}$} \, \left(\partial_1\phi,\mbox{$c\sqrt{\frac{\sqrt{3}}{S}}$}\partial_2\psi_{a,c}-\mbox{$\frac{a}{\sqrt{\sqrt{3}S}}$}\partial_1\psi_{a,c}\right)
\mbox{$  - \frac{1}{\sqrt{\sqrt{3}S}}$} \, \left(\partial_1\psi_{a,c},\mbox{$c\sqrt{\frac{\sqrt{3}}{S}}$}\partial_2\phi - \mbox{$\frac{a}{\sqrt{\sqrt{3}S}}$} \partial_1\phi\right)  
  \\
 &\qquad + \alpha \, \mbox{$ c\sqrt{\frac{\sqrt{3}}{S}}$} \, (\phi,\dot\psi_{a,c})_0 
  + \alpha \, \mbox{$\sqrt{\frac{\sqrt{3}}{S}}$} \frac{\sqrt{c^2(a+c)^2+S^2}}{2c} \, (\phi,\dot\psi_{a,c})_1 
  + \alpha \, \mbox{$\sqrt{\frac{\sqrt{3}}{S}}$} \frac{\sqrt{c^2(a-c)^2+S^2}}{2c}\, (\phi,\dot\psi_{a,c})_2
  \\
  &\qquad\qquad+ \alpha \, \mbox{$\sqrt{\frac{\sqrt{3}}{S}}$} \frac{c(a+c)}{2\sqrt{c^2(a+c)^2+S^2}}\, (\phi,\psi_{a,c})_1 
  + \alpha \, \mbox{$\sqrt{\frac{\sqrt{3}}{S}}$} \frac{c(a-c)}{2\sqrt{c^2(a-c)^2+S^2}}\, (\phi,\psi_{a,c})_2 
  \\
  &\qquad\qquad\qquad= \lambda_{a,c} (\phi,\dot\psi_{a,c}) + \dot\lambda_{a,c} (\phi,\psi_{a,c})
  \,,
\end{aligned}
\end{equation}
where we implicitly used that $\dot{\psi}_{a,c}\in H^1(\Omg_0)$; \cf~\cite[Prop. 2.3]{KL1}.
Choosing $\phi = \psi_{a,c}$ in~\eqref{id1} 
and combining the obtained identity
with~\eqref{id0} where we choose $\phi = \dot\psi_{a,c}$, we obtain
\begin{equation}\label{der1}
\begin{aligned}
\dot\lambda_{a,c} = -\mbox{$\frac{2}{\sqrt{\sqrt{3}S}}$} \, \left(\partial_1\psi_{a,c},\mbox{$c\sqrt{\frac{\sqrt{3}}{S}}$}\partial_2\psi_{a,c}-\mbox{$\frac{a}{\sqrt{\sqrt{3}S}}$}\partial_1\psi_{a,c}\right) 
  + \alpha \, \mbox{$ \sqrt{\frac{\sqrt{3}}{S}}$} \frac{c(a+c)}{2\sqrt{c^2(a+c)^2+S^2}} \, \|\psi_{a,c}\|_1^2 \\
  + \alpha \, \mbox{$\sqrt{\frac{\sqrt{3}}{S}}$} \frac{c(a-c)}{2\sqrt{c^2(a-c)^2+S^2}} \, \|\psi_{a,c}\|_2^2 
  \,.
  \end{aligned}
\end{equation}
By the symmetries of the equilateral triangle~$\Omega_0$,
it is easily seen that $\|\psi_0\|_0 = \|\psi_0\|_1 = \|\psi_0\|_2$.
Moreover, $\psi_0$ is an even function with respect to variable~$x$ (\cf~\eqref{eq:ef}).
Therefore,
\begin{equation}\label{symmetry}
  (\partial_1\psi_0,\partial_2\psi_0) = 0
  \,.
\end{equation}
Consequently,
\begin{equation}
  \dot\lambda_0 = 0  
  \,,
\end{equation}
so the equilateral triangle is a critical geometry 
with respect to variation of the geometric parameter~$a$. 

\subsection{The second derivative with respect to~\texorpdfstring{$a$}{a}} 
Differentiating~\eqref{der1} with respect to~$a$,
we get
\begin{equation}\label{der2}
\begin{aligned}
\ddot\lambda_{a,c} & = -\mbox{$\frac{2}{\sqrt{\sqrt{3}S}}$} \, \left(\partial_1\dot\psi_{a,c},\mbox{$c\sqrt{\frac{\sqrt{3}}{S}}$} \partial_2\psi_{a,c}-\mbox{$\frac{a}{\sqrt{\sqrt{3}S}}$}\partial_1\psi_{a,c}\right)\\
  &\qquad -\mbox{$\frac{2}{\sqrt{\sqrt{3}S}}$} \, \left(\partial_1\psi_{a,c},\mbox{$c\sqrt{\frac{\sqrt{3}}{S}}$}\partial_2\dot\psi_{a,c}-\mbox{$\frac{a}{\sqrt{\sqrt{3}S}}$}\partial_1\dot\psi_{a,c}-\mbox{$\frac{1}{\sqrt{\sqrt{3}S}}$}\partial_1\psi_{a,c}\right)\\
  &\qquad\quad+ \alpha \, \mbox{$\sqrt{\frac{\sqrt{3}}{S}}$} \frac{c(a+c)}{\sqrt{S^2+c^2(a+c)^2}} \, (\psi_{a,c},\dot\psi_{a,c})_1
  + \alpha \, \mbox{$\sqrt{\frac{\sqrt{3}}{S}}$}\frac{c(a-c)}{\sqrt{S^2+c^2(a-c)^2}} \, (\psi_{a,c},\dot\psi_{a,c})_2 \\
  &\qquad\qquad + \frac{c\alpha}{2} \, \mbox{$\sqrt{\frac{\sqrt{3}}{S}}$}
  \frac{S^2}{[c^2(a+c)^2+S^2]^{\frac{3}{2}}} \, \|\psi_{a,c}\|_1^2 
  + \frac{c\alpha}{2} \, \mbox{$\sqrt{\frac{\sqrt{3}}{S}}$} 
   \frac{S^2}{[c^2(a-c)^2+S^2]^{\frac{3}{2}}} \, \|\psi_{a,c}\|_2^2  
  \,.
\end{aligned}
\end{equation}
Putting $a=a_0=0$, $c = c_0 = \sqrt{\frac{S}{\sqrt{3}}}$ 
and using the symmetry 
$\|\psi_0\|_1^2 = \|\psi_0\|_2^2 = \frac{1}{3}\|\psi_0\|_{\sii(\partial\Omega_0)}^2$, 
we deduce
\begin{equation}\label{der2.0}
\begin{aligned}
  \ddot\lambda_0
  =-\mbox{$\frac{2}{\sqrt{\sqrt{3}S}}$} \, (\partial_1\dot\psi_0,\partial_2\psi_0)
  -\mbox{$\frac{2}{\sqrt{\sqrt{3}S}}$} \, (\partial_1\psi_0,\partial_2\dot\psi_0)
  + \frac{\alpha}{2} \,\mbox{$\sqrt{\frac{\sqrt{3}}{S}}$} \, (\psi_0,\dot\psi_0)_1
  - \frac{\alpha}{2} \, \mbox{$\sqrt{\frac{\sqrt{3}}{S}}$} \, (\psi_0,\dot\psi_0)_2
  \\
   + \alpha \, \mbox{$\frac{\sqrt{3}}{8S}$} 
  \, \|\psi_0\|_{\sii(\partial\Omega_0)}^2 
  + \mbox{$\frac{2}{\sqrt{3}S}$} \, \|\partial_1\psi_0\|^2 
  \,.
\end{aligned}
\end{equation}
Choosing $\phi=\dot\psi_{a,c}$ in~\eqref{id1} 
and putting $a=a_0=0, c=c_0=\sqrt{\frac{S}{\sqrt{3}}}$,
we have
\begin{equation}
\begin{aligned}
  \|\nabla\dot\psi_0\|^2
  - \mbox{$\frac{1}{\sqrt{\sqrt{3}S}}$} (\partial_1\dot\psi_0,\partial_2\psi_0)
  - \mbox{$\frac{1}{\sqrt{\sqrt{3}S}}$} (\partial_2\dot\psi_0,\partial_1\psi_0)
  + \alpha \, \|\dot\psi_0\|_{\sii(\partial\Omega_0)}^2  
  + \frac{\alpha}{4} \, \mbox{$\sqrt{\frac{\sqrt{3}}{S}}$} \, (\dot\psi_0,\psi_0)_1 \\
  - \frac{\alpha}{4} \, \mbox{$\sqrt{\frac{\sqrt{3}}{S}}$} \, (\dot\psi_0,\psi_0)_2 
  = \lambda_0 \|\dot\psi_0\|^2 
  \,.
\end{aligned}
\end{equation}
Using the variational characterisation of~$\lambda_0$, we obtain
\begin{equation}
   - \mbox{$\frac{1}{\sqrt{\sqrt{3}S}}$} (\partial_1\dot\psi_0,\partial_2\psi_0)
  - \mbox{$\frac{1}{\sqrt{\sqrt{3}S}}$} (\partial_2\dot\psi_0,\partial_1\psi_0)  
  + \frac{\alpha}{4} \, \mbox{$\sqrt{\frac{\sqrt{3}}{S}}$} \, (\dot\psi_0,\psi_0)_1 \\
  - \frac{\alpha}{4} \, \mbox{$\sqrt{\frac{\sqrt{3}}{S}}$} \, (\dot\psi_0,\psi_0)_2
  \leq 0
  \,.
\end{equation}
Substituting this estimate into~\eqref{der2.0}, we conclude with
\begin{equation}\label{second}
\begin{aligned}
  \ddot\lambda_0 
  &\leq
  \mbox{$\frac{2}{\sqrt{3}S}$} \, \|\partial_1\psi_0\|^2
  + \alpha \, \mbox{$\frac{\sqrt{3}}{8S}$}
  \, \|\psi_0\|_{\sii(\partial\Omega_0)}^2 
  \\
  &= 
  \mbox{$\frac{1}{\sqrt{3}S}$}
  \left(
  \|\nabla\psi_0\|^2
  + \alpha \, \mbox{$\frac{3}{8}$} 
  \, \|\psi_0\|_{\sii(\partial\Omega_0)}^2
  \right)
  \,,
\end{aligned}
\end{equation}
where the equality is obtained by dint of
the symmetry result $\|\partial_1\psi_0\| = \|\partial_2\psi_0\|$.
 
\subsection{The mixed second derivative} 
We next calculate the mixed derivative of $\lambda_{a,c}$ with respect to $a,c$ at the point $a = a_0 = 0$ and $c = c_0 = \sqrt{\frac{S}{\sqrt{3}}}$. 
We denote by the apostrophe the corresponding derivative with respect to~$c$. Putting $a=0$ in ~\eqref{der1}, we obtain
\begin{equation}\label{second0}
\begin{aligned}
\dot\lambda_{0,c}\! =\! -\mbox{$\frac{2}{\sqrt{\sqrt{3}S}}$} \, \left(\partial_1\psi_{0,c},\mbox{$c\sqrt{\frac{\sqrt{3}}{S}}$}\partial_2\psi_{0,c}\right) 
  + \alpha \, \mbox{$ \sqrt{\frac{\sqrt{3}}{S}}$} \frac{c^2}{2\sqrt{c^4+S^2}} \, \|\psi_{0,c}\|_1^2
  + \alpha \, \mbox{$\sqrt{\frac{\sqrt{3}}{S}}$} \frac{-c^2}{2\sqrt{c^4+S^2}} \, \|\psi_{0,c}\|_2^2 
  \,.
  \end{aligned}
\end{equation}
The symmetry of the domain $\Omega_{0,c}$ easily implies that $\psi_{0,c}(x,y)=\psi_{0,c}(-x,y)$ for all $(x,y)\in\Omega_{0,c}$, that $ (\partial_1\psi_{0,c},\partial_2\psi_{0,c})=0$, $\|\psi_{0,c}\|_1 = \|\psi_{0,c}\|_2$ and then, ${\dot\lambda_{0,c}} =0 $ for all $c > 0$. Moreover, as $a,c$ are independent, we get that 
$$
  \dot\lambda_0'=0
  .
$$

\subsection{The first derivative with respect to~\texorpdfstring{$c$}{c}} 
Differentiating ~\eqref{id0} with respect to $c$, we have

\begin{equation}\label{id1c}
\begin{aligned}
&-\mbox{$\frac{2S}{\sqrt{3}c^3}$} \, (\partial_1\psi_{a,c},\partial_1\phi)+\mbox{$\frac{S}{\sqrt{3}c^2}$} \, (\partial_1\phi, \partial_1\psi_{a,c}')
   + \left(\mbox{$\sqrt{\frac{\sqrt{3}}{S}}$} \partial_2\phi,
 \mbox{$c\sqrt{\frac{\sqrt{3}}{S}}$} \partial_2\psi_{a,c}-\mbox{$\frac{a}{\sqrt{\sqrt{3}S}}$} \partial_1\psi_{a,c}\right)\\
&\quad+ \left(\mbox{$c\sqrt{\frac{\sqrt{3}}{S}}$}\partial_2\phi -\mbox{$\frac{a}{\sqrt{\sqrt{3}S}}$}\partial_1\phi, \mbox{$c\sqrt{\frac{\sqrt{3}}{S}}$}\partial_2\psi_{a,c}'-\mbox{$\frac{a}{\sqrt{\sqrt{3}S}}$}\partial_1\psi_{a,c}'+\mbox{$\sqrt{\frac{\sqrt{3}}{S}}$}\partial_2\psi_{a,c}\right) 
  \\
  &\quad+ \alpha \, \mbox{$c\sqrt{\frac{\sqrt{3}}{S}}$} \, (\phi,\psi_{a,c}')_0
  + \alpha \, \mbox{$\sqrt{\frac{\sqrt{3}}{S}}$} \frac{\sqrt{c^2(a+c)^2+S^2}}{2c} \, (\phi,\psi_{a,c}')_1 
  + \alpha \, \mbox{$\sqrt{\frac{\sqrt{3}}{S}}$} \frac{\sqrt{c^2(a-c)^2+S^2}}{2c} \, (\phi,\psi_{a,c}')_2
  \\
  &\quad+ \alpha \, \mbox{$\sqrt{\frac{\sqrt{3}}{S}}$} \, (\phi, \psi_{a,c})_0 
  + \frac{\alpha}{2} \, \mbox{$\sqrt{\frac{\sqrt{3}}{S}}$} \frac{c^3(a+c)-S^2}{c^2\sqrt{c^2(a+c)^2+S^2}} (\phi, \psi_{a,c})_1 
  + \frac{\alpha}{2} \, \mbox{$\sqrt{\frac{\sqrt{3}}{S}}$} \frac{c^3(-a+c)-S^2}{c^2\sqrt{c^2(-a+c)^2+S^2}} (\phi, \psi_{a,c})_2
  \\
  &\qquad\qquad\qquad= \lambda_{a,c} (\phi,\psi_{a,c}') + \lambda_{a,c}' (\phi,\psi_{a,c})
  \,.
\end{aligned}
\end{equation}
Choosing $\phi=\psi_{a,c}$ in ~\eqref{id1c} and combining with ~\eqref{id0} where we choose $\phi=\psi_{a,c}'$, we get
\begin{equation}\label{der1c}
\begin{aligned}
\lambda_{a,c}' = -\mbox{$\frac{2S}{\sqrt{3}c^3}$} \, \|\partial_1\psi_{a,c}\|^2 + 2\left(\mbox{$\sqrt{\frac{\sqrt{3}}{S}}$}\partial_2\psi_{a,c},\mbox{$c\sqrt{\frac{\sqrt{3}}{S}}$}\partial_2\psi_{a,c} -\mbox{$\frac{a}{\sqrt{\sqrt{3}S}}$}\partial_1\psi_{a,c}\right) 
  + \alpha \, \mbox{$\sqrt{\frac{\sqrt{3}}{S}}$} \, \|\psi_{a,c}\|_0^2
   \\
  + \alpha \, \mbox{$ \sqrt{\frac{\sqrt{3}}{S}}$} \frac{c^3(a+c)-S^2}{2c^2\sqrt{c^2(a+c)^2+S^2}} \, \|\psi_{a,c}\|_1^2 
  + \alpha \, \mbox{$\sqrt{\frac{\sqrt{3}}{S}}$} \frac{c^3(-a+c)-S^2}{2c^2\sqrt{c^2(a-c)^2+S^2}} \, \|\psi_{a,c}\|_2^2 
  \,.
  \end{aligned}
\end{equation}
Plugging $a = a_0 = 0$ and $c= c_0 = \sqrt{\frac{S}{\sqrt{3}}}$ into the above formula we find that
\[
	\lambda_0' = 0.
\]
Hence, the equilateral triangle is also critical with respect to variation of the parameter $c$.

\subsection{The second derivative with respect to~\texorpdfstring{$c$}{c}} 
Differentiating the identity ~\eqref{der1c} with respect to variable $c$, we have
\begin{equation}\label{der2c}
\begin{aligned}
\lambda_{a,c}'' &= -\mbox{$\frac{4S}{\sqrt{3}c^3}$} \, (\partial_1\psi_{a,c},\partial_1\psi_{a,c}') +
\mbox{$\frac{6S}{\sqrt{3}c^4}$}\|\partial_1\psi_{a,c}\|^2+ 
   2\left(\mbox{$\sqrt{\frac{\sqrt{3}}{S}}$}\partial_2\psi_{a,c}',\mbox{$c\sqrt{\frac{\sqrt{3}}{S}}$}\partial_2\psi_{a,c}-\mbox{$\frac{a}{\sqrt{\sqrt{3}S}}$}\partial_1\psi_{a,c}\right)\\
   &\quad+2\left(\mbox{$\sqrt{\frac{\sqrt{3}}{S}}$}\partial_2\psi_{a,c},\mbox{$c\sqrt{\frac{\sqrt{3}}{S}}$}\partial_2\psi_{a,c}'-\mbox{$\frac{a}{\sqrt{\sqrt{3}S}}$}\partial_1\psi_{a,c}'+ \mbox{$\sqrt{\frac{\sqrt{3}}{S}}$}\partial_2\psi_{a,c}\right)\\
  &\quad\quad+ 2\alpha \, \mbox{$\sqrt{\frac{\sqrt{3}}{S}}$} \, (\psi_{a,c},\psi_{a,c}')_0
   + \alpha \, \mbox{$\sqrt{\frac{\sqrt{3}}{S}}$} \frac{c^3(a+c)-S^2}{c^2\sqrt{c^2(a+c)^2+S^2}} \, (\psi_{a,c},\psi_{a,c}')_1 \\
   &\quad\quad\quad+ \alpha \, \mbox{$\sqrt{\frac{\sqrt{3}}{S}}$} \frac{c^3(-a+c)-S^2}{c^2\sqrt{c^2(-a+c)^2+S^2}} \, (\psi_{a,c},\psi_{a,c}')_2\\
   &\quad\quad\quad\quad+ \frac{\alpha}{2} \, \mbox{$\sqrt{\frac{\sqrt{3}}{S}}$} \left(\frac{c^3(a+c)-S^2}{c^2\sqrt{c^2(a+c)^2+S^2}}\right)'  \, \|\psi_{a,c}\|_1^2 
+ \frac{\alpha}{2} \, \mbox{$\sqrt{\frac{\sqrt{3}}{S}}$}\left(\frac{c^3(-a+c)-S^2}{c^2\sqrt{c^2(-a+c)^2+S^2}}\right)'  \, \|\psi_{a,c}\|_{2}^2    
  \,.
\end{aligned}
\end{equation}
Using the equation~\eqref{id1c} with $\phi=\psi_{a,c}'$, 
$a=a_0=0$, and $c=c_0=\sqrt{\frac{S}{\sqrt{3}}}$, we obtain 
\begin{equation}\label{1lc}
\begin{aligned}
  \|\nabla\psi_0'\|^2
  - 2\mbox{$\sqrt{\frac{\sqrt{3}}{S}}$} (\partial_1\psi_0',\partial_1\psi_0)
  + 2\mbox{$\sqrt{\frac{\sqrt{3}}{S}}$} (\partial_2\psi_0',\partial_2\psi_0)
  + \alpha \, \|\psi_0'\|_{\sii(\partial\Omega_0)}^2  
  + \alpha \, \mbox{$\sqrt{\frac{\sqrt{3}}{S}}$} \, (\psi_0',\psi_0)_1 \\
  - \frac{\alpha}{2} \, \mbox{$\sqrt{\frac{\sqrt{3}}{S}}$} \, (\psi_0',\psi_0)_1
  - \frac{\alpha}{2} \, \mbox{$\sqrt{\frac{\sqrt{3}}{S}}$} \, (\psi_0',\psi_0)_2
  = \lambda_0 \|\psi_0'\|^2 
  \,.
\end{aligned}
\end{equation}
Combining ~\eqref{der2c} in which we put 
$a=a_0=0, c=c_0=\sqrt{\frac{S}{\sqrt{3}}}$ with ~\eqref{1lc}, we deduce
\begin{equation}\label{der2c0}
\begin{aligned}
\lambda_0''= 2\big(\lambda_0 \|\psi_0'\|^2-\alpha \, \|\psi_0'\|_{\sii(\partial\Omega_0)}^2 -\|\nabla\psi_0'\|^2\big)+ 4 \mbox{$\frac{\sqrt{3}}{S}$} \, \|\nabla\psi_0\|^2 +\alpha \, \mbox{$\frac{3\sqrt{3}}{2S}$} \|\psi_0\|_{\sii(\partial\Omega_0)}^2 
  \,.
\end{aligned}
\end{equation}
By using the variational characterisation of $\lambda_0$, we estimate
\begin{equation}\label{second2}
\lambda_0'' \leq 4 \mbox{$\frac{\sqrt{3}}{S}$} \, (\|\nabla\psi_0\|^2 +\alpha \, \mbox{$\frac{3}{8}$} \|\psi_0\|_{\sii(\partial\Omega_0)}^2 ).
\end{equation}

\subsection{Conclusions from the Hessian estimates} 
Now we are in position to formulate and prove the main result of this section on local optimality of equilateral triangles for moderate boundary parameters. The proof relies on showing that for certain values of the boundary parameter~$\aa$ the upper bounds on the second-order partial derivatives of the eigenvalue given in~\eqref{second} and~\eqref{second2} are negative. For $\aa$ sufficiently large by absolute value these bounds cease to be negative and in order to prove local optimality of the equilateral triangle finer estimates are necessary,
which we leave as an open problem in this paper.

Let us reformulate Theorem~\ref{Thm.local} from the introduction as follows.
\begin{Theorem}\label{thm:local}
There exists $\alpha_0= \alpha_0(S) < 0$ such that 
the equilateral triangle is a strict local maximiser of the lowest Robin eigenvalue in the family of triangles of area $S$ for the
boundary parameters $\alpha \in [\alpha_0,0)$. Moreover, for a fixed area $S$ and any
$\alpha\in [\alpha_0,0)$ there exist a sufficiently small $\eps > 0$ and a constant $C > 0$ such that
\begin{equation}\label{eq:ev_estimate}
	\lm_{a,c} 
	\le 
	\lm_0 - 
	C\left(a^2 +(c-c_0)^2\right)\,,
\end{equation}
provided that $|a| +|c-c_0| < \eps$.
\end{Theorem}
\begin{proof}
In order to show that the equilateral triangle is a strict local maximiser and that the bound in~\eqref{eq:ev_estimate} holds we rely
on the fact that $\dot\lambda_0 = \lambda_0' = \dot\lambda_0' = 0$ and on the estimates~\eqref{second} and~\eqref{second2}. The problem reduces
to finding the values of~$\alpha$ for which
\[
	f(\alpha): =\mbox{$\frac{1}{3}$} \|\nabla\psi_0\|^2 +\alpha \, \mbox{$\frac{1}{8}$} \|\psi_0\|_{\sii(\partial\Omega_0)}^2 <0.
\]
Thanks to the condition $\|\psi_0\| =1$ it follows from 
\cite[Chap.~4, Eq. (4.12)]{AH}
that $(\lambda_0)_\alpha' =\|\psi_0\|^2_{L^2(\partial\Omega_0)} >0 $, where $(\lambda_0)_\alpha'$ denotes the derivative of $\lambda_0$ with respect to
the boundary parameter $\alpha$.

From the variational characterisation of $\lambda_0$, we have
\begin{equation*}
\lambda_0  = \|\nabla \psi_0\|^2 + \alpha \|\psi_0\|^2_{L^2(\partial\Omega_0)}.
\end{equation*}
Hence, we obtain that
\begin{equation*}
\begin{aligned}
\mbox{$\frac{1}{3}$}\|\nabla\psi_0\|^2 +\alpha \, \mbox{$\frac{1}{8}$} \|\psi_0\|_{\sii(\partial\Omega_0)}^2&= \dfrac{1}{3} \lambda_0 - \frac{5}{24} \alpha \, \|\psi_0\|^2_{L^2(\partial\Omega_0)}
= \frac{\lambda_0}{3} - \frac{5\alpha}{24} (\lambda_0)_\alpha'.
\end{aligned}
\end{equation*}
Therefore, the following equivalence takes place
\begin{equation*}
f(\alpha) < 0\qquad \Longleftrightarrow\qquad \dfrac{(\lambda_0)_\alpha'}{\lambda_0} > \dfrac{8}{5\alpha}.
\end{equation*}
In addition, we deduce from ~\eqref{lambda0} that 
$(\lambda_0)_\alpha' = -\mbox{$\frac{8K}{\sqrt{3}S}$} K_\alpha'$. 
Recalling in addition~\eqref{alpha}, 
we get the following chain of equivalences
\[
f(\alpha) < 0
\qquad\Longleftrightarrow\qquad 
k(\alpha) <\dfrac{4}{5}
\qquad\Longleftrightarrow\qquad \dfrac{A}{K+A} <\dfrac{4}{5}\qquad
\Longleftrightarrow
\qquad A- 4K < 0.
\]
Eventually we obtain that
\begin{equation}\label{g}
 f(\alpha) < 0\qquad\Longleftrightarrow\qquad 
 g(t): =
 \dfrac{t}{1-t^2}+ \frac{t}{2(1-\frac{t^2}{4})} - 4\arctanh t -4 \arctanh \dfrac{t}{2} < 0,
\end{equation}	
where $t\in(0,1)$. 
It is easily seen that $g(t)$
is negative for small~$t$, while $g(t) \to +\infty$ as $t \to 1$
(see Figure~\ref{fig:graph_g}).
\begin{figure}[h!]
	\begin{center}
		\includegraphics[width=0.45\textwidth]{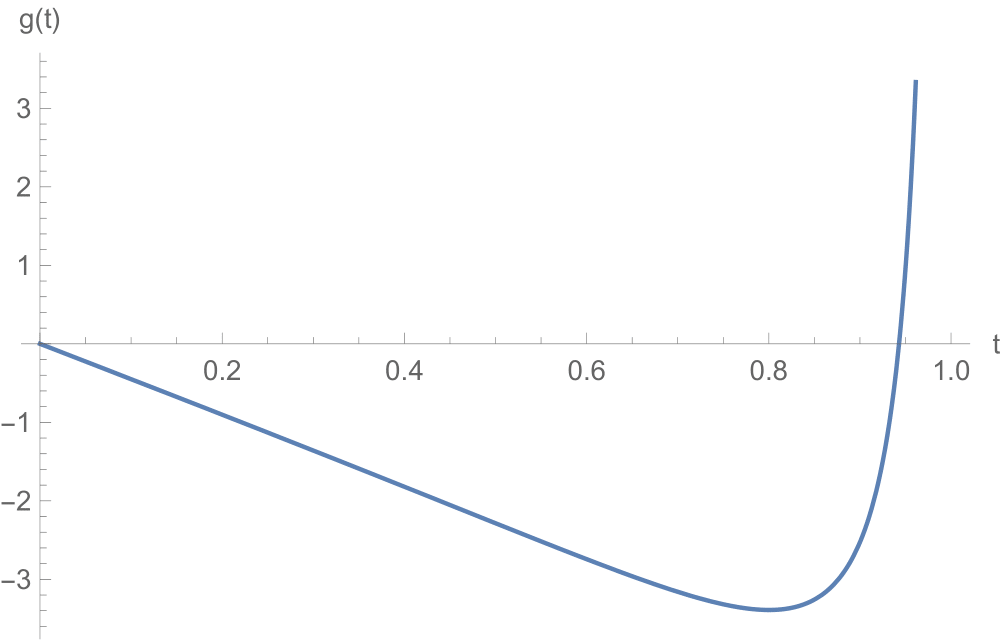}
		\caption{\small 
			The graph of the function $g$ defined in~\eqref{g}.}
		\label{fig:graph_g}
	\end{center}
\end{figure}
To get a quantitative estimate on the necessary smallness of~$t$, 
we use the elementary estimate $\arctanh t > t$ valid for all $t\in(0,1)$.
Then we deduce that $f(\alpha) < 0$ provided that 
$$
  t \leq t_0 := \frac{\sqrt{9-\sqrt{33}}}{2}.
$$
Using that $t=t(\alpha)$ is a decreasing function of $\alpha < 0$ 
and recalling~\eqref{alpha}, we get   
\begin{equation}\label{g.bis}
\begin{aligned}
	t \leq t_0 
	&\qquad\Longleftrightarrow\qquad 
	\alpha \geq -\mbox{$\frac{1}{\sqrt{\sqrt{3}S}}$} \, 
	t_0\left(\arctanh t_0+ \arctanh\dfrac{t_0}{2}\right)
	\\
	&\qquad\Longleftarrow\qquad
	\alpha \geq 
	-\frac{3}{2}\frac{t_0^2}{\sqrt{\sqrt{3}S}}.
\end{aligned}	
\end{equation}	
Hence, we also find using the exact value of $t_0$ that
\[
	\alpha \ge -\frac{0.92}{\sqrt{S}}\qquad\Longrightarrow\qquad t\le t_0.
\]
This concludes the proof of the theorem.
\end{proof}
\begin{Remark}
The present proof yields $\alpha_0(S) \leq -\frac{0.92}{\sqrt{S}}$,
so the bound is very quantitative, moreover it
provides an explicit dependence on the fixed area~$S$.
In particular, $\alpha_0(S) \to -\infty$ as $S \to 0$.
Let us also remark that a numerical analysis of the function~$g$
defined in~\eqref{g} yields that $g(t) < 0$ if, and only if,
$t < \tilde{t}_0 \approx 0.943$.
Plugging this value to the expression on the right-hand side 
of the first line of~\eqref{g.bis},
we would get an improved bound $\alpha_0(S) \leq -\frac{1.63}{\sqrt{S}}$.
\end{Remark}

To see that Theorem~\ref{Thm.local} follows 
as a consequence of Theorem~\ref{thm:local},
it is enough to notice that 
$|\Omega_{a,c}\setminus\Omega_0| \to 0$
implies that $a \to a_0 = 0$ and $c \to c_0 =\sqrt{\frac{S}{\sqrt{3}}}$.
In particular, the Hausdorff distance of $\Omega_{a,c}$ to $\Omega_0$
diminishes in the limit.

In the next corollary we provide a counterpart of Theorem~\ref{thm:local} 
under the fixed perimeter instead of the area constraint.
\begin{Corollary}\label{cor:local}
Let $\alpha_0= \alpha_0(S)$ be the negative constant from Theorem~\ref{thm:local}. 
The equilateral triangle is a strict local maximiser of the lowest Robin eigenvalue in the family of triangles of fixed perimeter $2\sqrt{3\sqrt{3}S}$ for the
	boundary parameters $\alpha \in (\alpha_0,0)$.
\end{Corollary}
\begin{proof}
	Let the triangle $\Omg_{a,c}$ be defined as before with the perimeter $|\p\Omg_{a,c}|$. We assume that $\Omg_{a,c}$ is not equilateral. Let us introduce the parameter 
	\[	
	\gamma_{a,c} := \frac{|\p\Omg_0|}{|\p\Omg_{a,c}|}.
	\]
	By the isoperimetric inequality for triangles we infer that $\gamma_{a,c}\in(0,1)$.
	Clearly, the mapping 
	$\dR\times\dR_+\ni(a,c)\mapsto \gamma_{a,c}\Omg_{a,c}$
	 parameterises generic triangles of the same perimeter as $\Omg_0$. In particular, we have $|\gamma_{a,c}\Omg_{a,c}| \le |\Omg_0| = S$. Hence, for all $\aa\in(\aa_0,0)$ we have by Theorem~\ref{thm:local} that there exists an open neighbourhood $\cU\subset\dR^2$ of the point $(0, \sqrt{\tfrac{S}{\sqrt{3}}})$ such that for any $(a,c)\in\cU$ holds
	 \begin{equation}\label{eq:isop1}
	 	\lm_1^\aa(\gamma_{a,c}\Omg_{a,c}) \le \lm_1^\aa(\gamma_{a,c}\Omg_0).
	 \end{equation}
	 On the other hand by Proposition~\ref{prop:monotone_S}
	 \begin{equation}\label{eq:isop2}
		 \lm_1^\aa(\gamma_{a,c}\Omg_0) < \lm_1^\aa(\Omg_0).
	 \end{equation}
	 Combining the inequalities~\eqref{eq:isop1} and~\eqref{eq:isop2} we get the claim.
\end{proof}
\begin{Remark}\label{rem:large_coupling}
The local optimality result of Theorem~\ref{thm:local} is worth complementing by the consequence of the large coupling asymptotics for the Robin Laplacian on a triangle. It follows from~ \cite[Thm.~2.3]{11} that
\[
	\lambda_{a,c} 
	\underset{\alpha \rightarrow -\infty}{\sim} -\frac{\alpha^2}{ {\sin^{2}\frac{\theta_\star}{2}}},
\]	
where $\theta_\star \in (0,\frac{\pi}{3}]$ is the magnitude of the smallest angle of the triangle $\Omega_{a,c}$. Hence, 
for any couple of parameters $(a,c) \in \mathbb{R}\times (0, \infty)$
there exists $\alpha_\star = \alpha_\star(a,c) < 0$ such that $\lambda_{a,c}\leq \lambda_0$ for all $\alpha\le\alpha_\star$.
\end{Remark}
%


\section{Small couplings}\label{sec:small}
%
In this section we show the validity of Conjecture~\ref{Conj.triangle} 
for negative~$\aa$ having small absolute value and under certain restriction on the parameters $a,c$ characterising the general triangle. We work with the equivalent quadratic form on the equilateral triangle constructed in Subsection~\ref{sec:geom} and use the ground state for the Robin Laplacian on an equilateral triangle as the trial function in the variational definition of the lowest eigenvalue. 
This section is complemented by a numerical computation of 
a region of limitation of the usage of this trial function 
for the proof of Conjecture~\ref{Conj.triangle} with $\aa <0$.

\subsection{The ground state of the equilateral triangle as a trial function}
Using the unitary transform~$U_{a,c}$ from Subsection~\ref{sec:geom}, 
we know that 
\begin{equation}\label{variational.hat}
  \lambda_{a,c}(\alpha) :=  \lambda_1^\alpha(\Omega_{a,c}) =
  \inf_{\stackrel[\psi\not=0]{}{\psi \in H^1(\Omega_0)}}
  \frac{\displaystyle
  \hat{h}_{\alpha,a,c}[\psi]}
  {\displaystyle
  \ \|\psi\|_{\sii(\Omega_0)}^2}
  \,,
\end{equation}
where the form~$\hat{h}_{\alpha,a,c}$ is given in~\eqref{form.hat}.
Employing the ground state of the equilateral triangle~$\psi_0$ as a trial function in~\eqref{variational.hat},
we obtain the upper bound
\begin{equation*} 
  \lambda_{a,c}(\alpha) \leq  \lambda_0(\alpha)
  + \frac{\displaystyle
  \hat{h}_{\alpha,{a,c}}[\psi_0] - \hat{h}_{\alpha}[\psi_0] }
  {\displaystyle
  \|\psi_0\|_{\sii(\Omega_0)}^2},
\end{equation*}
where $\hat{h}_\alpha$ is the abbreviation for the quadratic form $\hat{h}_{\alpha,a_0,c_0}$.
Using~\eqref{symmetry} we get
\begin{equation*} 
\begin{aligned}
  \hat{h}_{\alpha,a,c}[\psi_0] &- \hat{h}_{\alpha}[\psi_0] 
  = \left(\frac{c^2\sqrt{3}}{S} + \frac{S}{\sqrt{3}c^2}+\frac{a^2}{\sqrt{3}S}-2\right)  \|\partial_1\psi_0\|_{\sii(\Omega_0)}^2  \\
  &+ \alpha \,\left(\mbox{$\sqrt{\frac{\sqrt{3}}{S}}$} c 
 + \mbox{$\sqrt{\frac{\sqrt{3}}{S}}$}\frac{\sqrt{S^2+c^2(a+c)^2}}{2c} 
  +  \mbox{$\sqrt{\frac{\sqrt{3}}{S}}$} \frac{\sqrt{S^2+c^2(-a+c)^2}}{2c}
  - 3
  \right)
 \frac{\|\psi_0\|_{\sii(\partial\Omega_0)}^2}{3}
  \,.
 \end{aligned} 
\end{equation*}
Hence, for given $\aa < 0$, $a\in\dR$ and $c> 0$, 
the isoperimetric inequality in Conjecture~\ref{Conj.triangle} holds if 
the difference $\hat{h}_{\alpha,a,c}[\psi_0] - \hat{h}_{\alpha}[\psi_0]$ is non-positive.

In the proof of the main result of this section we use the following well-known geometric fact, which we recall for the convenience of the reader.
\begin{Proposition}[{\cite[Sec.~4]{12} or~\cite[\S~2.2]{K}}]\label{isoperimetric}
\
\begin{itemize}
\item [{\rm (i)}] Among triangles having the same area, the equilateral is the unique minimiser of the perimeter.
\item [{\rm(ii)}] Among triangles having the same area and the same base, the isosceles is the unique minimiser of the perimeter.
\end{itemize} 
\end{Proposition}
Denote the perimeter of the triangle $\Omega_{a,c}$ by 
$l(a) := |\partial \Omega_{a,c}|$ and recall 
the explicit formula~\eqref{explicit}. 
As a consequence of Proposition~\ref{isoperimetric}\,(ii), we obtain 
\begin{equation}\label{eq:iso_triangle_2}
 \underset{a\in{\mathbb{R}}}{\min}\, l(a)= l(0)=  2c + 2 \sqrt{c^2+ \frac{S^2}{c^2}}. 
\end{equation}

Now we are in position to formulate and prove the main result of this section.
\begin{Theorem}\label{Thm.small}
Let $A,B$ and~$M$ be positive numbers such that $B > A$.
There exists a constant $\alpha_{\rm c} =\alpha_{\rm c}(M, A, B,S) < 0$ such that
$$
  \lambda_{a,c} < \lambda_0
$$ 
for all $\alpha\in [\alpha_{\rm c}, 0)$, $|a|\le M$ and $c\in[A,B]$ such that $(a,c)\ne \left(0,\sqrt{\tfrac{S}{\sqrt{3}}}\right)$.  
\end{Theorem}
\begin{proof}
Throughout this proof the ground state $\psi_0$ for the equilateral triangle is assumed to be normalised to~$1$ in $\sii(\Omega_0)$.
Let us introduce three auxiliary functions
\[
z(a,c): = \dfrac{f_1(a) - 3}{\frac{c^2\sqrt{3}}{S} + \frac{S}{\sqrt{3}c^2}+\frac{a^2}{\sqrt{3}S}-2}
, \qquad
f_1(a): = \sqrt{\frac{\sqrt{3}}{S}} \ \frac{l(a)}{2}
\qquad\text{and}\qquad
g_1(\alpha) := \frac{3 \|\nabla \psi_0\|^2_{L^2(\Omega_0)}}{-2\alpha \, \|\psi_0\|_{\sii(\partial\Omega_0)}^2},
\] 
where $l(a)$ is the notation for the perimeter of~$\Omega_{a,c}$
introduced above~\eqref{eq:iso_triangle_2}
and explicitly given in~\eqref{explicit}.
Note that the denominator of~$z$ is positive by the Cauchy inequality.
It is easy to show the limiting properties
\[
\underset{(a,c)\in{\mathbb{R}}\times (0,\infty)}{\inf} z(a,c) =  \lim\limits_{a\rightarrow \pm\infty} z(a,c)=\lim\limits_{c\rightarrow \infty}z(a,c)= 0.
\]
Using the symmetry of the equilateral triangle, we have the following equivalences
\begin{equation*}
\hat{h}_{\alpha,a,c}[\psi_0] - \hat{h}_{\alpha}[\psi_0] \leq 0 \qquad\Longleftrightarrow\qquad
\frac{3\|\partial_1\psi_0\|_{\sii(\Omega_0)}^2}{-\alpha\|\psi_0\|_{\sii(\partial\Omega_0)}^2}
\leq z(a,c)
\qquad
\Longleftrightarrow\qquad g_1(\alpha) \le z(a,c).
\end{equation*}

Let us now look at the limiting properties of~$g_1$ as $\alpha \to 0$.
By the normalisation of $\psi_0$, 
we have $ \lambda_0= \|\nabla\psi_0\|^2_{L^2(\Omega_0)} + \alpha \|\psi_0\|^2_{L^2(\partial\Omega_0)}$.
Note that $\lambda_0=\lambda_0(\alpha)$ 
converges in the limit $\alpha\rightarrow0$ to the first eigenvalue of the Neumann Laplacian on $\Omega_0$, which is equal to zero.
Moreover, one has (see, \eg, \cite[Chap.~4, Eq.~(4.12)]{AH})
\begin{equation}\label{DN1}
\dfrac{\der\lambda_0}{\der\alpha} 
= \|\psi_0(\alpha)\|^2_{L^2(\partial\Omega_0)} > 0, 
\end{equation}
for every~$\alpha$, 
where by writing $\psi_0(\alpha)$ we indicate the dependence of $\psi_0$ on the boundary parameter $\alpha$.
It implies that
\begin{equation}\label{DN2}
\dfrac{\der\lambda_0}{\der\alpha}\bigg|_{\alpha=0} = \lim\limits_{\alpha \rightarrow 0}\|\psi_0(\alpha)\|^2_{L^2(\partial\Omega_0)}=\dfrac{|\partial \Omega_0|}{|\Omega_0|}.
\end{equation}
In addition, by the definition of the derivative of $\lambda_0$ with respect to $\alpha$ at the point $\alpha = 0$, we get
\begin{equation}\label{1}
\dfrac{|\partial \Omega_0|}{|\Omega_0|} =
\lim\limits_{\alpha \rightarrow 0}\frac{\lambda_0(\alpha)-\lambda_0(0)}{\alpha}= \lim\limits_{\alpha \rightarrow 0} \frac{\|\nabla\psi_0(\alpha)\|^2_{L^2(\Omega_0)} + \alpha \|\psi_0(\alpha)\|^2_{L^2(\partial\Omega_0)}}{\alpha} . 
\end{equation}
Combining \eqref{1} and \eqref{DN2}, we obtain 
\begin{equation*}
\lim\limits_{\alpha \rightarrow 0} g_1(\alpha)=0.
\end{equation*}

By Proposition~\ref{isoperimetric}\,(i), 
we have the strict inequality
$\underset{a\in{\mathbb{R}}}{\min} f_1(a) > 3 $ 
unless $a = a_0 = 0$ and $c = c_0 = \sqrt{\frac{S}{\sqrt{3}}}$
(in which case we have equality). 
It follows that $z(a,c) > 0$ unless $ a = a_0$ and $c = c_0$.
Moreover, $(a,c) \mapsto z(a,c)$ is a continuous function.
Let $\varepsilon > 0$ be arbitrary and we denote by $\mathcal{B}_\varepsilon$ the open disk of radius~$\varepsilon$ with the centre $(0,c_0)$.
Hence, for any $M > 0$ and $B > A > 0$  there exists a constant $\hat\alpha_{\rm c} = \hat\alpha_{\rm c}(M,A,B,\varepsilon) < 0$ such that  for all $\alpha\in [\hat\alpha_{\rm c}, 0)$
one has 
\[
	0< g_1(\alpha)< \underset{([-M, M]\times [A,B])\setminus\mathcal{B}_\varepsilon}{\min} z(a,c),
\] 
which implies that
\[
	\lambda_{a,c} < \lambda_0,\qquad
	\forall (a,c) \in ([-M, M]\times [A,B])\setminus\mathcal{B}_\varepsilon .
\]
By Theorem~\ref{thm:local}, 
for a sufficiently small $\varepsilon > 0$ and for $\alpha \in [\alpha_0,0)$, 
we have
\[
\lambda_{a,c} < \lambda_0,\qquad\forall (a,c) \in 
\mathcal{B}_\varepsilon \setminus \{(0,c_0)\}.  
\]
Choosing $\alpha_{\rm c} := \max\{\hat\alpha_{\rm c},\alpha_0\}$
we get for all $\alpha\in [\alpha_{\rm c},0)$ that
\[
\lambda_{a,c} < \lambda_0,\qquad \forall (a,c) \in \left([-M, M]\times [A,B]\right)\setminus\{(0,c_0)\}. 
\]
This concludes the proof of the theorem.
\end{proof}

\subsection{Limitations of the trial function}
In the rest of this section we restrict to the special case 
$c=S=\frac{1}{\sqrt{3}}$ and perform an analysis of 
the region in the $(\aa,a)$-plane for which the present choice 
of the ground state of the equilateral triangle 
as a trial function fails to prove Conjecture~\ref{Conj.triangle} 
with $\aa <0$.
In this special case, the ground state~$u_0$ 
of the equilateral triangle~\eqref{eq:ef} 
and its partial derivative with respect to the first variable are given by
\[
\begin{aligned}
    u_0
    &=
    \cosh(L-2Ky)+ 2 \cosh(M+Ky) \cosh(\sqrt{3}Kx),\\
	\partial_1 u_0 & 
	= 
	2\sqrt{3}K\cosh(Ky + M)\sinh(K\sqrt{3}x) .	
\end{aligned}
\]
Hence, we can express the squared norm of $\partial_1 u_0$ as follows:
\begin{equation*}
\begin{aligned}
\|\partial_1 u_0\|^2 &=12K^2\int_0^1\int_{\frac{y}{\sqrt{3}}-\frac{1}{\sqrt{3}}}^{-\frac{y}{\sqrt{3}}+\frac{1}{\sqrt{3}}} \cosh^2(Ky + M)\sinh^2(K\sqrt{3}x)\der x \der y \\
&= \frac{\sqrt{3}}{8} \big( -4-8K^2+4\cosh(2K)+\cosh(2K-2M)+4\cosh(2M)\\
&\qquad-5\cosh(2K+2M)+8K\sinh(2M)+4K\sinh(2K+2M)\big) \, .
\end{aligned}
\end{equation*}
Similarly, the square of the norm of the trace of~$u_0$ 
on $\partial\Omega_0$ can be computed as follows:
\begin{equation*}
\begin{aligned}
&\|u_0\|^2_{L^2(\partial\Omega_0)} = 3\|u_0\|^2_0 = 3\int_{-\frac{1}{\sqrt{3}}}^{\frac{1}{\sqrt{3}}} [\cosh(L)+ 2\cosh(M)\cosh(\sqrt{3}Kx)]^2 \der x  \\
&\quad= \frac{\sqrt{3}}{K} \big( 3K+K\cosh(2L)+2K\cosh(2M)+8\cosh(L)\cosh(M)\sinh(K)+2\cosh^2(M)\sinh(2K)\big) \, .
\end{aligned}
\end{equation*}
Using the above expressions for 
$\|\partial_1 u_0\|^2$ and $\|u_0\|^2_{L^2(\partial\Omega_0)}$
we plot for $c = S = \frac{1}{\sqrt{3}}$ in
Figure~\ref{Fig.test1} the region where  
$
  \delta(\alpha,a) :=
  \hat{h}_{\alpha,a,\frac{1}{\sqrt{3}}}[u_0] - \hat{h}_{\alpha}[u_0]
$ 
is non-positive.
This plot demonstrates that Conjecture~\ref{Conj.triangle} 
for negative $\alpha$ holds
for every~$a$ by choosing~$|\alpha|$ sufficiently small
(for larger~$a$, smaller~$|\alpha|$ needed).
However, it also shows the limitations of the present 
choice of the trial function 
(for any negative~$\alpha$, there exists a sufficiently large positive~$a$
such that the difference $\delta(\alpha,a)$ is positive).
In summary, for eccentric triangles (\ie~$a$ large), 
a different choice of trial function is necessary.

\begin{figure}[h]
\begin{center}
\includegraphics[width=0.45\textwidth]{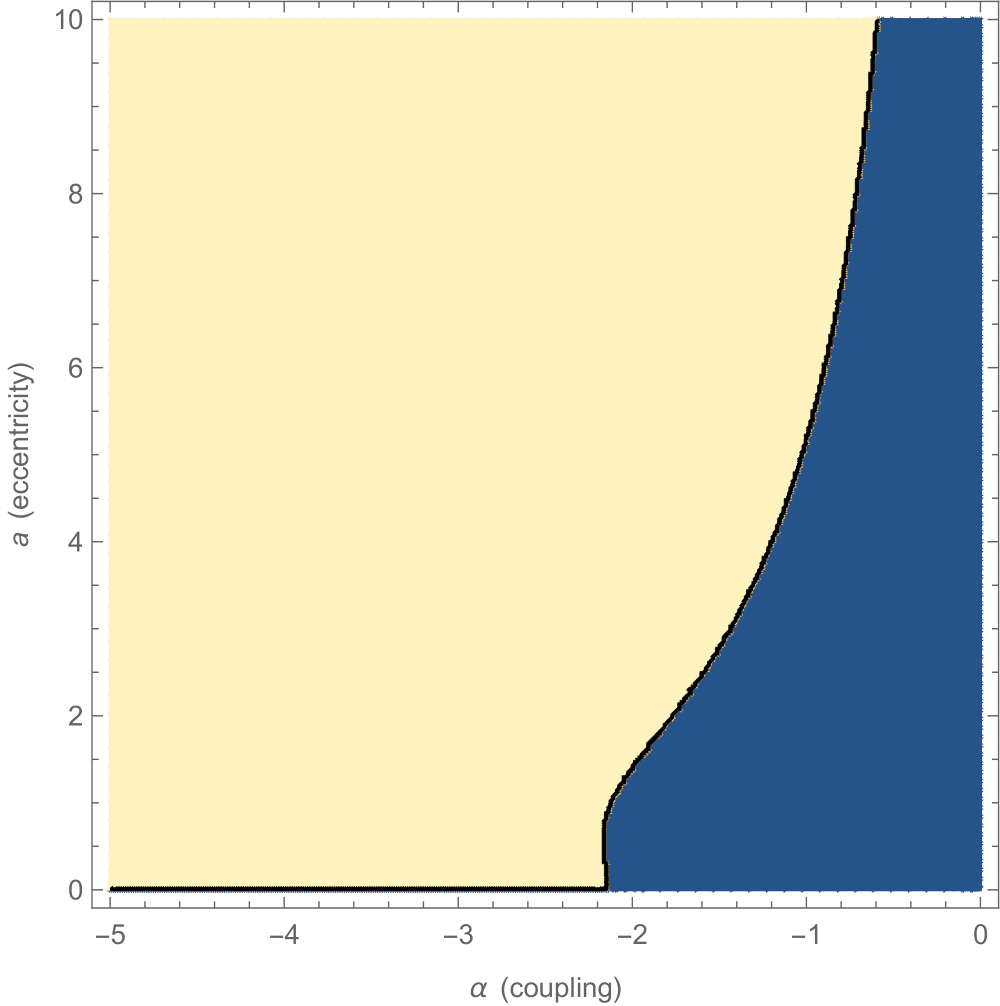}
\caption{\small The blue colour indicates 
the region of validity of Conjecture~\ref{Conj.triangle}: \newline
the equilateral triangle eigenfunction
as a trial function.}
\label{Fig.test1}
\end{center}
\end{figure}
%

\section{Large couplings}\label{sec:large}
%
In this section we show
the validity of Conjecture~\ref{Conj.triangle}
for negative $\alpha$ with larger values of~$|\alpha|$. 
To this aim one needs to use different trial functions, 
which better reflect 
the behaviour of the eigenfunction in a general triangle
in this large coupling limit. 
As in the preceding section,
analytical results are supplemented by numerical computations of 
the region of limitations of these trial functions
to establish the validity of
Conjecture~\ref{Conj.triangle} with negative $\aa$.  

\subsection{The Neumann ground state as a trial function}
We start with the constant function as a trial function.
\begin{Theorem}\label{Thm.large}
For any $\alpha <0$, there exist positive constants 
$a_1=a_1(\alpha)$, $c_1 = c_1(\alpha)$ and $c_2=c_2(\alpha)$ 
such that 
$$ 
  \lambda_{a,c} < \lambda_0
$$ 
holds under any of the following restrictions:
\begin{enumerate}
\item[(i)]
$|a| > a_1$ and $c >0$;
\item[(ii)]
$a\in\mathbb{R}$ and $c> c_1$;
\item[(iii)]
$a\in \mathbb{R}$ and $c< c_2$.
\end{enumerate}
\end{Theorem}
\begin{proof}
It follows from the 
variational characterisation~\eqref{variational.hat} with
the trial function being the characteristic function~$\mathbbm{1}$ 
of the triangle $\Omega_0$ that
\begin{equation*}
\lambda_{a,c} 
\le  \frac{\hat{h}_{\alpha, a,c}[\mathbbm{1}]}{\|\mathbbm{1}\|^{2}}
= \alpha  f_2(a,c) 
\qquad \mbox{with} \qquad
f_2(a, c):= \frac{|\partial\Omega_{a,c}|}{S} .
\end{equation*}
Recall the explicit formula~\eqref{explicit} 
for the perimeter of~$\Omega_{a,c}$ and that we are dealing
with the area constraint, so that $S = |\Omega_0|$.
Obviously, the desired inequality $\lambda_{a,c}< \lambda_0$ is satisfied if
\begin{equation}\label{eq}
\begin{aligned}
f_2(a,c) > \frac{\lambda_0}{\alpha} .
\end{aligned}
\end{equation}
By Proposition~\ref{isoperimetric}\,(ii), 
\[ 
f_2(a,c)\geq\underset{a\in{\mathbb{R}}}{\min} f_2(a,c)= f_2(0,c)= \frac{2}{S}\left(c+\frac{\sqrt{c^4+S^2}}{c}\right)\, .
\]
Since
$$
  \lim\limits_{c\rightarrow 0^+} \min_{a\in\mathbb{R}} f_2(a,c) 
  = +\infty
  \,,
$$ 
there exists $c_2 =c_2 (\alpha)$ such that $\lambda_{a,c} < \lambda_0$ 
for all $(a,c)\in \mathbb{R}\times (0,c_1)$;
this establishes condition~(iii).
At the same time, since
$$
  \lim\limits_{|a|\rightarrow \infty} \min_{c >0} f_2(a,c)= \lim\limits_{c \rightarrow \infty}\min_{a\in\mathbb{R}} f_2(a,c) = +\infty  
$$ 
for any fixed $\alpha <0$,
there exist $ a_1=a_1(\alpha) >0, c_1=c_1(\alpha) >0$ such that 
$ f_2(a,c) > \frac{\lambda_0}{\alpha}$  
for all $(|a|,c) \in (a_1, \infty)\times(0,\infty) $ 
or for all $(a,c)\in \mathbb{R}\times (c_1,\infty)$;
this establishes conditions~(i) and~(ii).
\end{proof}

\begin{Remark}
The constants~$a_1$ and~$c_1$ (respectively, the constant~$c_2$)
obtained in our proof above 
tend to~$+\infty$ (respectively, tends to~$0$)
as $\alpha \to -\infty$.
Indeed, it is enough to recall~\eqref{eq} and notice that
$\frac{\lambda_0}{\alpha}\rightarrow +\infty$ as $\alpha\rightarrow-\infty$
(\cf~ Remark~\ref{rem:large_coupling}), 
while $f_2(a,c)$ is independent of~$\alpha$.
\end{Remark}
\begin{figure}[h]
\begin{center}
\includegraphics[width=0.45\textwidth]{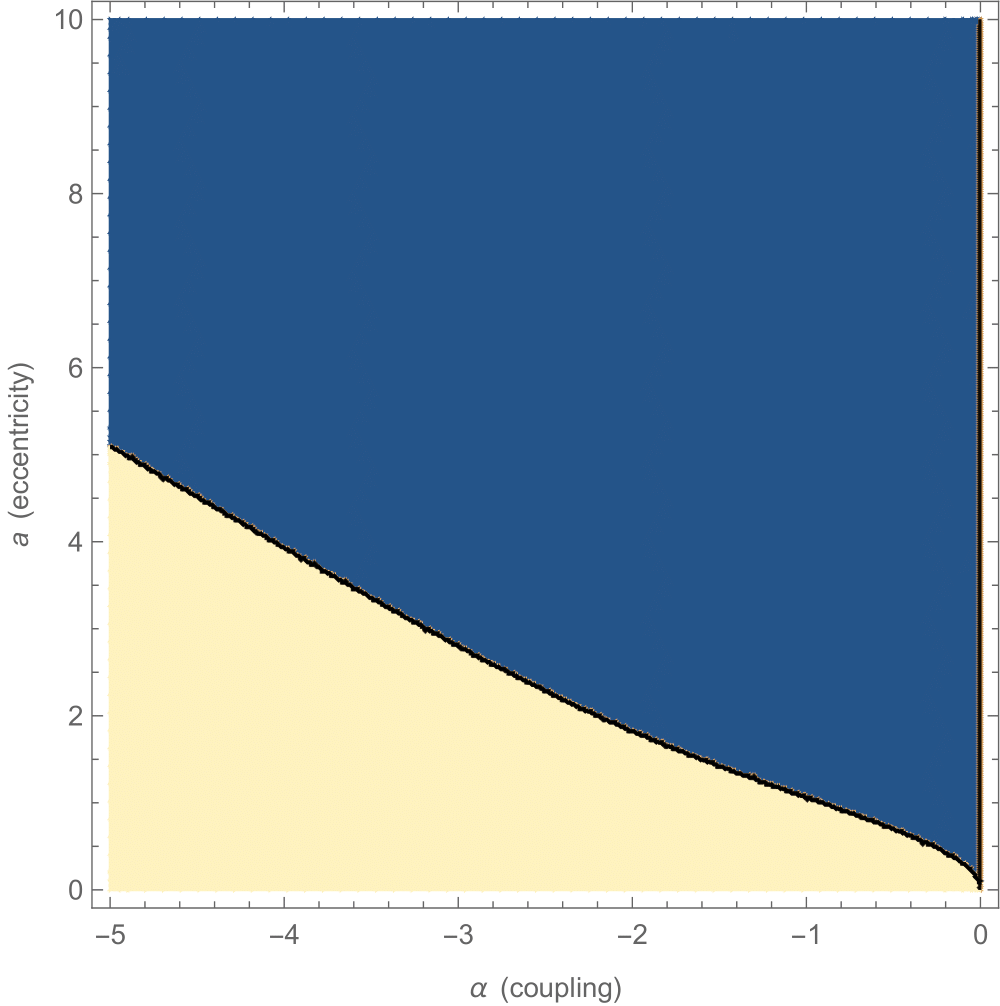}
\caption{\small 
The blue colour indicates the region 
of validity of Conjecture~\ref{Conj.triangle}: \newline 
the constant function as a trial function.}
\label{Fig.test-wrong}
\end{center}
\end{figure}

The analysis based on the constant test function can be supplemented 
by a numerical evidence.
Figure~\ref{Fig.test-wrong} plots the region for $c = S=\frac{1}{\sqrt{3}}$ where $h_\alpha[\mathbbm{1}]/\|\mathbbm{1}\|_{\sii(\Omega_{a,\frac{1}{\sqrt{3}}})}^2 
- \lambda_0(\alpha)$ 
is non-positive.
Now, given any negative~$\alpha$, 
we are able to cover all sufficiently eccentric triangles.

\subsection{The ground state of a sector as a trial function}
The ground state of the Robin Laplacian on a convex sector with a negative boundary parameter is explicitly known 
(see~\cite[Ex.~2.5, Lem.~2.6]{LP08} and also~\cite[Thm.~2.3\,(b)]{KP18}).
 In the next theorem we use a truncation of this ground state as a trial function and find a sufficient condition
	in terms of the boundary parameter, the area of the triangle, the smallest angle and the length of the smaller
	side of the triangle adjacent to this angle for the isoperimetric inequality to hold.
	This theorem provides also a quantitative version of the observation made in Remark~\ref{rem:large_coupling}.
\begin{Theorem}\label{thm:large}
	Let the parameters $a\in\mathbb{R}$ and $c >0$ be such that the triangle $\Omega_{a,c}$ is not equilateral. 
	Let $\theta_\star\in (0,\pi/3)$ be the smallest angle of $\Omega_{a,c}$. Let $L' >0$ be the length of the smaller side of $\Omega_{a,c}$ adjacent to that angle. Assume that $\alpha <0$ is such that
	\begin{equation}\label{eq:condition}
		-4\alpha^2 + \frac{24\alpha}{\mbox{$\sqrt{\sqrt{3}S}$}}
		-\frac{36}{\sqrt{3}S} \ge
		-
		\frac{\alpha^2}{\sin^2(\theta_\star/2)}
		\big(1-2\exp(2\alpha L'\cot(\theta_\star/2))\big).
	\end{equation}
	Then the inequality 
	$$
	  \lambda_{a,c} < \lambda_0
	$$ 
	holds. In particular,
	for a fixed triangle $\Omega_{a,c}$ the condition~\eqref{eq:condition} holds for all $|\alpha|$ large enough
	and for a fixed $\alpha < 0$ this condition holds for all $\theta_\star$ small enough.
\end{Theorem}
\begin{proof}
	We choose the vertex of the triangle
	corresponding to the angle $\theta_\star$ as the origin and introduce the orthogonal coordinate system $(x',y')$ so that the $x'$-axis coincides with the bisector line of the triangle~$\Omega_{a,c}$ emerging from its smallest angle. In this coordinate system we introduce the function
	\begin{equation}\label{eq:test_func}
		u_\star(x',y') := \exp\left(\frac{\alpha x'}{\sin(\theta_\star/2)}\right)
		.
	\end{equation}
Obviously, $u_\star \in H^1(\Omega_{a,c})$,
so it is an admissible trial function.	
	By a direct computation we get the estimates 
	\begin{equation}\label{eq:estimates}
	\begin{aligned}
		\|u_\star\|^2_{L^2(\Omega_{a,c})} 
		&\le \int_{-\frac{\theta_\star}{2}}^{\frac{\theta_\star}{2}}\int_0^{\infty}\exp\left(\frac{2\alpha r\cos\theta}{\sin(\theta_\star/2)} \right)r\der r \der \theta
		=
		\frac{\sin^2(\theta_\star/2)}{4\alpha^2}
		\int_{-\frac{\theta_\star}{2}}^{\frac{\theta_\star}{2}}\frac{1}{\cos^2\theta}\der\theta 
		\\
		&= \frac{\sin^2(\theta_\star/2)\tan(\theta_\star/2)}{2\alpha^2}\,, \\[0.3ex]
		\|u_\star\|^2_{L^2(\partial\Omega_{a,c})}
		&\ge
		2\int_0^{L'}\exp\left(
		2\alpha r\cot(\theta_\star/2)\right)\der r 
		\\
		&=
		-\frac{\tan(\theta_\star/2)}{\alpha}
		\Big(1-\exp\left(2\alpha L'\cot(\theta_\star/2)\right)\Big)\,.
	\end{aligned}	
	\end{equation}
	Moreover, we easily find that
	\begin{equation}\label{eq:norm_gradient}
		\|\nabla u_\star\|^2_{L^2(\Omega_{a,c})} = \frac{\alpha^2}{\sin^2(\theta_\star/2)}\|u_\star\|^2_{L^2(\Omega_{a,c})}\,.
	\end{equation}
	Combining~\eqref{eq:estimates} and~\eqref{eq:norm_gradient}, 
	we get from 
	variational characterisation~\eqref{variational.hat} with
    the trial function $u_\star$ that
	\begin{equation}\label{eq:lmac_bnd}	
	\begin{aligned}
		\lambda_{a,c}& \le
		\frac{\|\nabla u_\star\|^2_{L^2(\Omega_{a,c})}+ \alpha\|u_\star\|^2_{L^2(\partial\Omega_{a,c})}}{\|u_\star\|^2_{L^2(\Omega_{a,c})}}\\
			&\le 
		\frac{\alpha^2}{\sin^2(\theta_\star/2)}
		-
		\frac{2\alpha^2}{\sin^2(\theta_\star/2)}\big(1-\exp(2\alpha L'\cot(\theta_\star/2))\big)\\
		&=
		-
		\frac{\alpha^2}{\sin^2(\theta_\star/2)}\big(1-2\exp(2\alpha L'\cot(\theta_\star/2))\big).
	\end{aligned}	
	\end{equation}
	Relying on the analysis in Subsection~\ref{ssec:equilateral}, 
	for the equilateral triangle we have
	$\lambda_0= -\frac{4K^2}{\sqrt{3}S}$.
	Here the parameter $K$ can be estimated as follows
	\[
	\begin{aligned}
	K&= \arctanh\left(\mbox{$\sqrt{\sqrt{3}S}$}\dfrac{-\alpha}{K}\right)+ \arctanh\left(\mbox{$\sqrt{\sqrt{3}S}$}\dfrac{-\alpha}{2K}\right)\\
	&= \frac{1}{2}\ln\left(\frac{(1+t)(1+\frac{t}{2})}{(1-t)(1-\frac{t}{2})}\right)
	= \frac{1}{2}\ln\left( 1+\frac{6t}{t^2-3t+2}\right) 
	\\
	&< \frac{1}{2} \frac{6t}{t^2-3t+2}
	 < \frac{6}{t-2} +\frac{3}{1-t} < \frac{3}{1-t},\\
	\end{aligned}
	\]
	where we have used that $t=-\sqrt{\sqrt{3}S}\frac{\aa}{K}\in(0,1)$.
	Hence, we get that $K < 3-\alpha \sqrt{\sqrt{3}S}$ and thus
	\[
	\lambda_0 > -
	\frac{4(3-\alpha \sqrt{\sqrt{3}S})^2}{ \sqrt{3}S} = -4\alpha^2 + \frac{24\alpha}{\mbox{$\sqrt{\sqrt{3}S}$}}
	-\frac{36}{\sqrt{3}S}.
	\]
	The desired claim follows upon combination of 
	the above estimate with~\eqref{eq:lmac_bnd}.

	The first additional observation,
	that for a fixed triangle $\Omega_{a,c}$ the condition~\eqref{eq:condition} holds for all $|\alpha|$ large enough, follows from the fact that
	both the left- and the right-hand sides
	in~\eqref{eq:condition}
	tend to $-\infty$ as $\alpha\rightarrow-\infty$ 
	and their ratio tends to
	$4\sin^2(\theta_\star/2)\in(0,1)$.
	In order to verify the other additional observation, 
	it suffices to notice that by the triangle inequality and Proposition~\ref{isoperimetric}\,(i)
	we have $L' \ge l(a)/4 > \frac{3}{2}\sqrt{\frac{S}{\sqrt{3}}}$
	and hence for any fixed $\alpha < 0$ 
	the right-hand side  in~\eqref{eq:condition} tends to $-\infty$ as $\theta_\star\rightarrow 0$,
	 while the left-hand side is independent of $\theta_\star$. 
\end{proof}	

Theorem~\ref{Thm.global} from the introduction
is a special version of a combination of 
Theorems~\ref{Thm.small} and~\ref{thm:large}. 

\subsection{Numerical support}
We complement Theorem~\ref{thm:large}
by numerical computations in the special case $c = S = \frac{1}{\sqrt{3}}$. In order to perform these computations we require some extra analysis.
For $a \in[0,\frac{2}{\sqrt{3}}]$ the smallest angle of the triangle $\Omega_{a,\frac{1}{\sqrt{3}}}$ is at the vertex $(-\frac{1}{\sqrt{3}},0)$ while for $a > \frac{2}{\sqrt{3}}$ the smallest angle of this triangle is at the vertex $(a,1)$.
We will analyse these two cases separately.

Let $a\in[0,\frac{2}{\sqrt{3}}]$. Using elementary geometric arguments we find that
\[
	\sin\theta_\star = \frac{1}{\sqrt{1+\left(a+\frac{1}{\sqrt{3}}\right)^2}}.
\]
Hence, we derive with the aid of trigonometric identities that 
\begin{equation}\label{eq:cotsin}
	\sin^2\left(\tfrac{\theta_\star}{2}\right) = 
	\frac{\sqrt{1+\left(a+\frac{1}{\sqrt{3}}\right)^2}- \left(a+\frac{1}{\sqrt{3}}\right) }{2\sqrt{1+(a+\frac{1}{\sqrt{3}})^2}}\quad\text{and}\quad
	\cot\left(\tfrac{\theta_\star}{2}\right) = \sqrt{1+\left(a+\tfrac{1}{\sqrt{3}}\right)^2}+a+\tfrac{1}{\sqrt{3}}.
\end{equation}
Moreover, the length of the shorter side of the triangle $\Omega_{a,\frac{1}{\sqrt{3}}}$ adjacent
to its smallest angle is $L'= \frac{2}{\sqrt{3}}$.

Let $a \in(\frac{2}{\sqrt{3}},\infty)$.
Again using elementary geometric arguments we find that
\[
	\sin\theta_\star = \frac{2}{\sqrt{3a^4+4a^2+\frac{16}{3}}}
\]
and derive
\[
\sin^2\left(\tfrac{\theta_\star}{2}\right) = 
\frac{1-\sqrt{\frac{9a^4+12a^2+4}{9a^4+12a^2+16}}}{2}\quad\text{and}\quad
\cot\left(\tfrac{\theta_\star}{2}\right) = \sqrt{
\frac{2}{1-\sqrt{\frac{9a^4+12a^2+4}{9a^4+12a^2+16}}}-1}
\,.
\]
Moreover, the length of the shorter side of the triangle $\Omega_{a,\frac{1}{\sqrt{3}}}$ adjacent
to its smallest angle is \[
L'= \sqrt{1+\left(a-\tfrac{1}{\sqrt{3}}\right)^2}.
\]
Using the above analysis of the two cases
$a\in[0,\frac{2}{\sqrt{3}}]$ and $a > \frac{2}{\sqrt{3}}$ we plot in Figure~\ref{fig:cond} the region where the condition~\eqref{eq:condition} is satisfied.
\begin{figure}[h!]
	\begin{center}
		\includegraphics[width=0.45\textwidth]{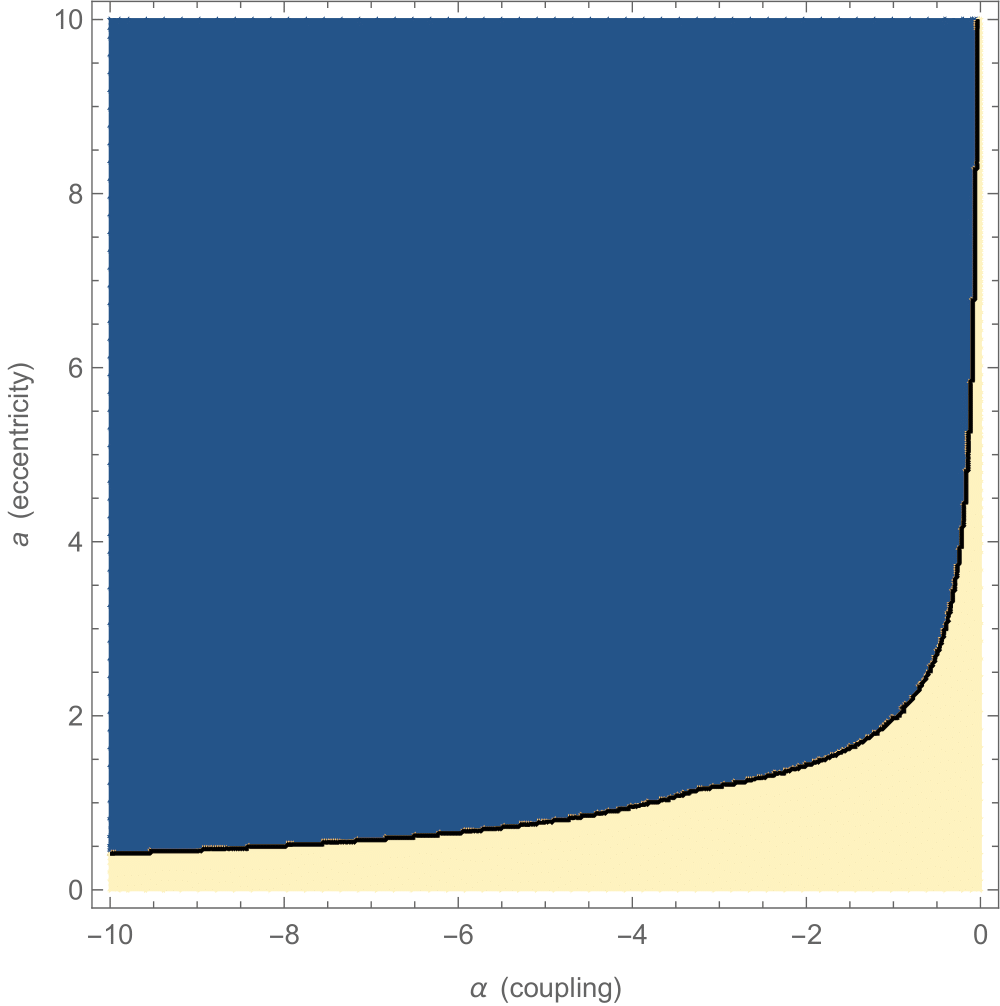}
		\caption{\small 
			The blue colour indicates the region 
			where the sufficient condition~\eqref{eq:condition}
			for validity of Conjecture~\ref{Conj.triangle} is satisfied.}
		\label{fig:cond}
	\end{center}
\end{figure}

Furthermore, we find numerically the region where the test function $u_\star$ defined in~\eqref{eq:test_func} yields the inequality in Conjecture~\ref{Conj.triangle}. In this analysis we make a simplification and construct the function $u_\star$ always based
on the angle at the vertex $(-\frac{1}{\sqrt{3}},0)$ even though this angle is not the smallest angle of the triangle $\Omega_{a,\frac{1}{\sqrt{3}}}$ for $a > \frac{2}{\sqrt{3}}$. As we will see from the numerical plot even after such a simplification we still obtain a large region of validity for Conjecture~\ref{Conj.triangle}.

In order to perform this numerical test it is convenient to express the function $u_\star$ as a function of initial coordinates $(x,y)$. Clearly, the distance between the points $(x,y)$ and $(-\frac{1}{\sqrt{3}},0)$ is given by
\[
	r = \sqrt{\left(x+\frac{1}{\sqrt{3}}\right)^2+y^2}.
\]
Let $\theta'$ be the magnitude of the angle formed by the vertices $(x,y)$, $(-\frac{1}{\sqrt{3}},0)$, and $(\frac{1}{\sqrt{3}},0)$. We immediately find that
\[
	\sin\theta' = \frac{y}{r},\qquad\cos\theta' = \frac{x+\tfrac{1}{\sqrt{3}}}{r}.
\]
Hence, we obtain using the expression for $\cot(\frac{\theta_\star}{2})$ in~\eqref{eq:cotsin} that
\[
\begin{aligned}
	\frac{x'}{\sin\left(\frac{\theta_\star}{2}\right)}& = 
	\frac{r\cos\left(\frac{\theta_\star}{2}-\theta'\right)}{\sin\left(\frac{\theta_\star}{2}\right)} = \left(x+\tfrac{1}{\sqrt{3}}\right)\cot\left(\frac{\theta_\star}{2}\right) + y\\
	&= 
	\left(x+\tfrac{1}{\sqrt{3}}\right)
	\left(\sqrt{1+\left(a+\tfrac{1}{\sqrt{3}}\right)^2}+a+\tfrac{1}{\sqrt{3}}\right) + y.
\end{aligned}
\]
Finally, we conclude that 
\begin{equation}\label{eq:ustar}
	u_\star(x,y) = \exp\left(
	\alpha\left[	\left(x+\tfrac{1}{\sqrt{3}}\right)
	\left(\sqrt{1+\left(a+\tfrac{1}{\sqrt{3}}\right)^2}+a+\tfrac{1}{\sqrt{3}}\right) + y\right]
	\right).
\end{equation}
\color{black}
\begin{figure}[h!]
\begin{center}
\includegraphics[width=0.45\textwidth]{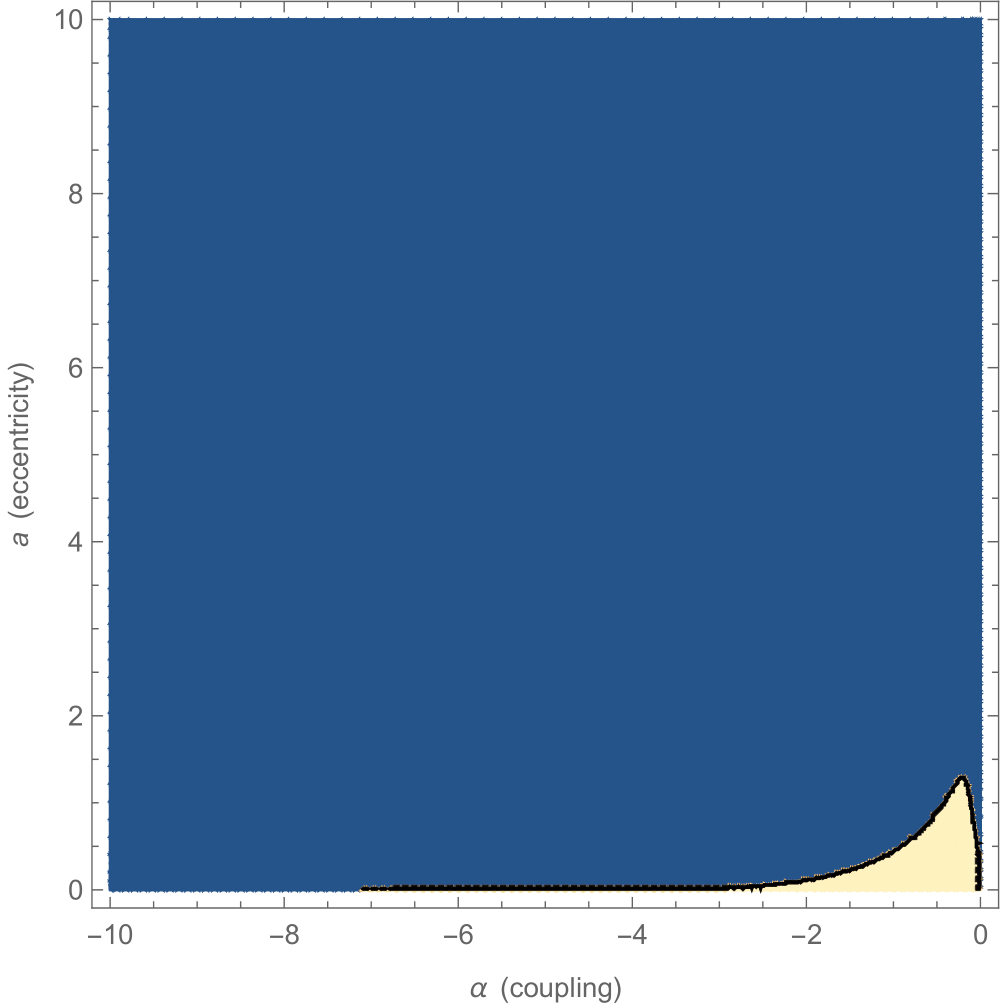}
\caption{\small 
	The blue colour indicates the region 
	of validity of Conjecture~\ref{Conj.triangle}: \newline 
the function~$u_\star$ in~\eqref{eq:ustar} as a trial function.}
	\label{Fig.test-wrong1}
\end{center}
\end{figure}
Figure~\ref{Fig.test-wrong1} plots the region where the difference
$h_\alpha[u_\star]/\|u_\star\|_{\sii(\Omega_{a,
		\frac{1}{\sqrt{3}}})}^2 
- \lambda_0(\alpha)$ is non-positive.
From this plot we see that
the test function $u_\star$
suffices to show that the inequality in Conjecture~\ref{Conj.triangle} holds for any $\aa < 0$ and all $a >0$ not too small and for any $a > 0$ and $\aa < 0$ with sufficiently large $|\aa|$. Moreover, according to this plot,
one can show using the trial function $u_\star$ that there exists $a_\star > 0$ such that the inequality in Conjecture~\ref{Conj.triangle} holds for any $\aa <0$ and all $a >a_\star$.

\subsection*{Acknowledgement}
The authors would like to express their gratitude 
to the American Institute of Mathematics (AIM) 
for a support to organise the workshop 
\emph{Shape optimization with surface interactions}
(San Jose, USA, 17--21 June 2019),
which stimulated the present research.
The first and last authors (D.K. and T.V.)
were partially supported by the EXPRO grant 
No. 20-17749X of the Czech Science Foundation (GA\v{C}R).
The second author (V.L.) acknowledges the support by 
the grant No.~21-07129S of the Czech Science Foundation (GA\v{C}R).
%


\end{document}